\documentclass[twoside]{article}
\usepackage{latexsym}
\usepackage{pstricks}
\usepackage{amssymb}
\usepackage{amsmath}
\usepackage{amsfonts}
\usepackage{amsthm,array}
\usepackage{array}
\usepackage{epsfig}
\usepackage{graphicx}
\usepackage[numbers,sort&compress]{natbib}

\newtheorem{theorem}{Theorem}[section]
\newtheorem{proposition}{Proposition}[section]
\newtheorem{lemma}{Lemma}[section]
\newtheorem{corollary}{Corollary}[section]

\newtheorem{definition}{Definition}[section]

\newtheorem{remark}{Remark}[section]
\newtheorem{example}{Example}[section]

\newtheorem{remark-definition}{Remark and Definition}[section]
\newtheorem{rem-not}{Remark and Notation}[section]

\begin{document}
\title{\bf Interval B-Tensors and Interval Double B-Tensors\thanks{The first author's work was supported by the National Natural Science Foundation of P.R. China (Grant No.12171064).}}
\date{}
\author{ Li Ye, Yisheng Song\thanks{Corresponding author E-mail: yisheng.song@cqnu.edu.cn}\\
	School of Mathematical Sciences,  Chongqing Normal University, \\
	Chongqing, 401331, P.R. China. \\ Email: neutrino1998@126.com (Ye); yisheng.song@cqnu.edu.cn (Song)}
\maketitle

{\noindent\bf Abstract.}
This paper systematically investigates the properties and characterization of interval B-tensors and interval double B-tensors. We propose verifiable necessary and sufficient conditions that allow for determining whether an entire interval tensor family belongs to these classes based solely on its extreme point tensors. The study elucidates profound connections between these interval tensors and other structured ones such as interval Z-tensors and P-tensors, while also providing simplified criteria for special cases like circulant structures. Furthermore, under the condition of even order and symmetry, we prove that interval B-tensors (double B-tensors) ensure the property of being an interval P-tensor. This work extends interval matrix theory to tensors, offering new analytical tools for fields such as polynomial optimization and complementarity problems involving uncertainty.

\vspace{.3cm}
{\noindent\bf Mathematics Subject Classification.} 90C23, 15A69, 15A72, 90C30, 65G40
\vspace{.3cm}

{\noindent\bf Keywords.} B-tensor, Double B-tensor, Interval tensor, P-tensor.

\section{Introduction}
\setcounter{section}{1}

In numerical analysis, uncertainty quantification, and robust optimization, matrices and tensors are fundamental tools for modeling and computation \cite{01,02,ql2017}. When matrix entries are subject to measurement errors, parameter fluctuations, or data uncertainty, the theory of interval matrices has been developed to study collective properties of families of matrices whose entries range within given intervals. Since its emergence in the 1960s, interval matrix theory has been widely applied in control theory, system reliability analysis, and robust computing, and has given rise to the study of various structured classes, such as interval P-matrices, interval M-matrices, interval B-matrices, and others \cite{a1,a2,a3,a4,a5,a6,a7,a8,lm2023}. In \cite{lm2023}, Lorenc systematically established a characterization theory for interval B-matrices and interval doubly B-matrices, providing equivalent criteria that depend only on the extreme point matrices, and studied their relationships with interval P-matrices and interval Z-matrices. This work not only provides important tools for the study of structured classes of interval matrices but also naturally motivates the extension to higher-order tensor cases.

As higher-order generalizations of matrices, tensors play important roles in polynomial optimization, complementarity problems, machine learning, and signal processing \cite{b1,b2,b3,b4,b5,b6}. Beheshti et al. \cite{bf2022} extended several classes of interval matrices to their tensor counterparts and probed into the corresponding properties and characterizations. Cui and Zhang \cite{CZ2023} derived bounds for the H-eigenvalues of even order real symmetric interval tensors as well as negative interval tensors.  Rahmati and Tawhid \cite{sr2020} proposed the slice-property of tensor sets by extending the row-property of matrix sets, and proved that a convex tensor set is slice-positive definite if and only if its extreme point set holds this property. Regarding Hurwitz stability, Bozorgmanesh et al. \cite{boz2020} established a link between the stability and positive definiteness of symmetric interval tensors, and gave the estimation of Z-eigenvalue ranges via constructing the tensor counterpart of Bendixson's theorem (consistent with its matrix-based expression), which only applies to symmetric interval tensors. Additional relevant works on interval tensors can be found in \cite{c1,c2,c3}.

B-tensors and double B-tensors, introduced in recent years, are important structural classes closely related to tensor positive definiteness, solvability of complementarity problems, eigenvalue localization, and numerical algorithm design \cite{sq2015,qs2014,pl2014,cy2015,d1,d2,d3,d4}. When tensor entries are subject to interval uncertainty, the definitions of interval B-tensors and interval double B-tensors arise naturally meaning that every tensor within the interval is of the corresponding type. However, unlike interval matrix theory which has well-established verification criteria, the study of interval B-tensors and interval double B-tensors is still in its early stages, lacking computable characterization criteria based solely on the extreme point tensors, analogous to those in interval matrix theory.

This paper aims to establish a theoretical framework for interval B-tensors and interval double B-tensors. We first formalize their definitions based on classical B-tensors and double B-tensors. Then, we establish several equivalent characterization theorems for interval B-tensors, revealing their connections with structures like interval diagonal dominance and interval Z-tensors. Furthermore, we investigate criteria for interval double B-tensors and analyze their relationship with interval B-tensors. Finally, for special structures such as circulant interval tensors, we provide simplified criteria.

The paper is organized as follows: Section 2 reviews B-tensors, double B-tensors, and related concepts. Section 3 systematically studies the properties and criteria of interval B-tensors. Section 4 discusses interval double B-tensors and their relationship with interval B-tensors. Section 5 concludes the paper and suggests future research directions.

\section{Preliminaries}
\setcounter{section}{2}


A real $m$th order $n$-dimensional tensor $\mathcal{A} = (a_{i_1i_2\cdots i_m})$ is a multi-array of real entries $a_{i_1i_2\cdots i_m}$, where each index $i_j$ belongs to $[n]:=\{1,2,\cdots,n\}$ for $j\in[m]:=\{1,2,\cdots,m\}$. The set of all such tensors in denoted by $T_{m,n}$. Entries of the form $a_{i_1i_1\cdots i_1}$ are termed diagonal entries, while others are off-diagonal entries. If the entries $a_{i_1i_2\cdots i_m}$ of a tensor $\mathcal{A}\in T_{m,n}$ remain unchanged under any permutation of their indices, then $\mathcal{A}$ is called a symmetric tensor. The set of all real $m$th order $n$-dimensional symmetric tensors is denoted by $S_{m,n}$.


\begin{definition}\label{def-bt}\textup{\cite{sq2015}}
A tensor $\mathcal{A} = (a_{i_1i_2\cdots i_m})\in T_{m,n}$ is said to be a B-tensor if it satisfies the following conditions for all $i_1\in[n]$:
\begin{itemize}
  \item [(a)]  $\sum\limits_{i_2,i_3,\cdots, i_m=1}^n a_{i_1i_2\cdots i_m}>0$,
  \item [(b)]  $\frac{1}{n^{m-1}}(\sum\limits_{i_2,i_3,\cdots, i_m=1}^n a_{i_1i_2\cdots i_m})>a_{i_1j_2\cdots j_m}~~\mbox{ for all }(j_2,j_3,\cdots, j_m)\neq(i_1,i_1,\cdots,i_1).$
\end{itemize}
\end{definition}

\begin{lemma}\label{r1}\textup{\cite{sq2015}}
Let $\mathcal{A} = (a_{i_1i_2\cdots i_m})\in T_{m,n}$ is a B-tensor, then for all $i_1\in[n]$
$$a_{i_1i_1\cdots i_1}>\max\{0,a_{i_1i_2\cdots i_m}:(i_2,i_3,\cdots, i_m)\neq(i_1,i_1,\cdots,i),i_2,i_3,\cdots,i_m\in[n]\}:=\gamma_{i_1}^+(\mathcal{A}).$$
\end{lemma}

\begin{proposition}\label{p2}\textup{\cite{sq2015}}
Let $\mathcal{A} = (a_{i_1i_2\cdots i_m})\in T_{m,n}$. Then $\mathcal{A}$ is a B-tensor if and only if for all $i_1\in[n]$
$$\sum\limits_{i_2,i_3,\cdots, i_m=1}^n a_{i_1i_2\cdots i_m}>n^{m-1}\gamma_{i_1}^+(\mathcal{A}).$$
\end{proposition}
The following result directly follows from Proposition \ref{p2}.
\begin{proposition}\label{p3}
Let $\mathcal{A} = (a_{i_1i_2\cdots i_m})\in T_{m,n}$. Then $\mathcal{A}$ is a B-tensor if and only if for all $i_1\in[n]$
$$a_{i_1i_1\cdots i_1}-\gamma_{i_1}^+(\mathcal{A})>\sum\limits_{(i_2,i_3,\cdots, i_m)\neq(i_1,i_1,\cdots,i_1)}(\gamma_{i_1}^+(\mathcal{A})-a_{i_1i_2\cdots i_m}).$$
\end{proposition}

\begin{definition}\label{def-dd}\textup{\cite{qs2014}}
A tensor $\mathcal{A} = (a_{i_1i_2\cdots i_m})\in T_{m,n}$ is said to be a diagonally dominated tensor if for all $i_1\in[n]$,
$$a_{i_1i_1\cdots i_1}\geq\sum\limits_{(i_2,i_3,\cdots, i_m)\neq(i_1,i_1,\cdots,i_1)}|a_{i_1i_2\cdots i_m}|.$$
$\mathcal{A} = (a_{i_1i_2\cdots i_m})\in T_{m,n}$ is said to be a strictly diagonally dominated tensor if for all $i_1\in[n]$,
$$a_{i_1i_1\cdots i_1}>\sum\limits_{(i_2,i_3,\cdots, i_m)\neq(i_1,i_1,\cdots,i_1)}|a_{i_1i_2\cdots i_m}|.$$
\end{definition}

\begin{proposition}\label{p4}\textup{\cite{sq2015}}
If $\mathcal{A} = (a_{i_1i_2\cdots i_m})\in T_{m,n}$ is a Z-tensor, then the following are equivalent.
\begin{itemize}
  \item [(a)] $\mathcal{A}$ is a B-tensor.
  \item [(b)] $\sum\limits_{i_2,i_3,\cdots, i_m=1}^n a_{i_1i_2\cdots i_m}>0$ for each $i_1\in[n]$.
  \item [(c)] $\mathcal{A}$ is strictly diagonally dominant.
\end{itemize}
\end{proposition}

\begin{definition}\label{def-zt}\textup{\cite{zq2014}}
A tensor $\mathcal{A} = (a_{i_1i_2\cdots i_m})\in T_{m,n}$ is said to be a Z tensor if for all its off diagonal entries $a_{i_1i_2\cdots i_m}\geq0$.
\end{definition}

For $\mathcal{A}= (a_{i_1\cdots i_m})\in T_{m,n}$ and $x\in\mathbb{R}^n$, $M=(m_{ij})\in\mathbb{R}^{n\times n}$, and $k\in[m]$, $\mathcal{A}\times_1 M\times_2M\times_3\cdots\times_m M$ is a tensor in $T_{m,n}$ which defined by
$$(\mathcal{A}\times_1 M\times_2M\times_3\cdots\times_m M)_{i_1i_2\cdots i_m}=\sum\limits_{j_1j_2\cdots j_m=1}^n a_{j_1j_2\cdots j_m}m_{i_1j_1}\cdots m_{i_mj_m}.$$
And for a vector $\mathbf{x}\in\mathbb{R}^n$, $\mathcal{A}\mathbf{x}^{m-1}$ is a vector in $\mathbb{R}^n$ with the $i$th component
$$(\mathcal{A}\mathbf{x}^{m-1})_i:=\sum\limits_{i_2,\cdots i_m=1}^n a_{ii_2\cdots i_m}x_{i_2}\cdots x_{i_m},$$
for $1\leq i \leq n$.

\begin{definition}\label{def-pt}\textup{\cite{sq2015}}
A tensor $\mathcal{A} = (a_{i_1i_2\cdots i_m})\in T_{m,n}$ is said to be a P tensor if for all $\mathbf{x}\in\mathbb{R}^n\setminus\{\mathbf{0}\}$, there is $i_1\in[n]$ such that $x_{i_1}^{m-1}(\mathcal{A}\mathbf{x}^{m-1})_i>0$.
\end{definition}

\begin{proposition}\label{p1}\textup{\cite{qs2014,pl2014}}
Suppose $\mathcal{A} \in T_{m,n}$ is a B-tensor with $m$ is even. Then $\mathcal{A}$ is a P tensor if $\mathcal{A}$ is a Z tensor or $\mathcal{A}$ is a symmetric tensor.
\end{proposition}

For a tensor $\mathcal{A} = (a_{i_1i_2\cdots i_m})\in T_{m,n}$, the set $\{a_{i_1i_2\cdots i_m}:i_2,i_3,\cdots, i_m\in[n]\}$ is called the $i_1$-th row of $\mathcal{A}$.

\begin{proposition}\label{p5}
Let $\mathcal{A},\mathcal{B},\mathcal{C} \in T_{m,n}$, where $\mathcal{A}$ and $\mathcal{B}$ are B-tensors. If for each $i_1\in[n]$, either
$c_{i_1i_1\cdots i_1}=a_{i_1i_1\cdots i_1}$ and the off-diagonal entries of the $i_1$-th row in $\mathcal{A}$ and $\mathcal{C}$ forms bijection,
or $c_{i_1i_1\cdots i_1}=b_{i_1i_1\cdots i_1}$ and the off-diagonal entries of the $i_1$-th row in $\mathcal{B}$ and $\mathcal{C}$ forms bijection holds, then $\mathcal{C}$ is also a B-tensor.
\end{proposition}
\begin{proof}
As Definition \ref{def-bt} contains no conditions linking different rows, we can combine rows that meet the requirements, with the diagonal entries remaining unchanged, to construct a B-tensor.
\end{proof}

\begin{definition}\label{def-dbt}\textup{\cite{cy2015}}
A tensor $\mathcal{A} = (a_{i_1i_2\cdots i_m})\in T_{m,n}$ is said to be a double B-tensor if it satisfies the following conditions for all $i_1\in[n]$:
\begin{itemize}
  \item [(a)] $a_{i_1i_1\cdots i_1}>\gamma_{i_1}^+(\mathcal{A})$,
  \item [(b)] $a_{i_1i_1\cdots i_1}-\gamma_{i_1}^+(\mathcal{A})\geq \sum\limits_{(i_2,i_3,\cdots, i_m)\neq(i_1,i_1,\cdots,i_1)} (\gamma_{i_1}^+(\mathcal{A})-a_{i_1i_2\cdots i_m}),$
  \item [(c)]  for all $j_1\neq i_1$
  \begin{align*}
  &(a_{i_1i_1\cdots i_1}-\gamma_{i_1}^+(\mathcal{A}))(a_{j_1j_1\cdots j_1}-\gamma_{j_1}^+(\mathcal{A}))\\
  &>(\sum\limits_{(i_2,i_3,\cdots, i_m)\neq(i_1,i_1,\cdots,i_1)} (\gamma_{i_1}^+(\mathcal{A})-a_{i_1i_2\cdots i_m}))(\sum\limits_{(j_2,j_3,\cdots, j_m)\neq(j_1,j_1,\cdots,j_1)} (\gamma_{j_1}^+(\mathcal{A})-a_{j_1j_2\cdots j_m})).
  \end{align*}
\end{itemize}
\end{definition}

\begin{proposition}\label{p8}\textup{\cite{cy2015}}
If $\mathcal{A}\in T_{m,n}$ is a B-tensor, then it is a double B-tensor.
\end{proposition}

\begin{proposition}\label{p9}\textup{\cite{cy2015}}
If $\mathcal{A}\in T_{m,n}$ is a symmetric double B-tensor with $m$ is even, then it is a P tensor.
\end{proposition}

\begin{proposition}\label{p10}\textup{\cite{cy2015}}
If $\mathcal{A}\in T_{m,n}$ is a double B-tensor, then precisely one of the following is satisfied.
\begin{itemize}
  \item [(a)]  $\mathcal{A}$ is a B-tensor, or
  \item [(b)]  there is a unique $j_1\in[n]$ such that
$$a_{j_1j_1\cdots j_1}-\gamma_{j_1}^+(\mathcal{A})=\sum\limits_{(j_2,j_3,\cdots,j_m)\neq(j_1,j_1,\cdots,j_1)}(\gamma_{j_1}^+(\mathcal{A})-a_{j_1j_2\cdots j_m}),$$
and for all $i_1\in[n]\setminus\{j_1\}$,
$$a_{i_1i_1\cdots i_1}-\gamma_{i_1}^+(\mathcal{A})> \sum\limits_{(i_2,i_3,\cdots,i_m)\neq(i_1,i_1,\cdots,i_1)}(\gamma_{i_1}^+(\mathcal{A})-a_{i_1i_2\cdots i_m}).$$
\end{itemize}
\end{proposition}

\begin{definition}\label{def-ct}\textup{\cite{cq2016}}
A tensor $\mathcal{A} = (a_{i_1i_2\cdots i_m})\in T_{m,n}$ is said to be a circulant tensor if $a_{j_1j_2\cdots j_m}\equiv a_{k_1k_2\cdots k_m}$ when $k_l=j_l+1$ for $j_l,k_l\in[n]$.
\end{definition}

\begin{proposition}\label{p11}
If $\mathcal{A}\in T_{m,n}$ is a circulant tensor, then then the following are equivalent.
\begin{itemize}
  \item [(a)] $\mathcal{A}$ is a B-tensor.
  \item [(b)] $\mathcal{A}$ is a double B-tensor.
  \item [(c)] $a_{11\cdots 1}-\gamma_{1}^+(\mathcal{A})> \sum\limits_{(i_2,i_3,\cdots,i_m)\neq(1,1,\cdots,1)}(\gamma_{1}^+(\mathcal{A})-a_{1i_2\cdots i_m}).$
\end{itemize}
\end{proposition}
\begin{proof}
``$(a)\Rightarrow(b)$''. That is a result of Proposition \ref{p8}.

``$(b)\Rightarrow(c)$''. Let $\mathcal{A}$ is a double B-tensor, then for any $i_1\in[n]\setminus\{1\}$,
\begin{align*}
  &(a_{11\cdots 1}-\gamma_{1}^+(\mathcal{A}))(a_{i_1i_1\cdots i_1}-\gamma_{i_1}^+(\mathcal{A}))\\
  &>(\sum\limits_{(i_2,i_3,\cdots, i_m)\neq(1,1,\cdots,1)} (\gamma_{1}^+(\mathcal{A})-a_{1i_2\cdots i_m}))(\sum\limits_{(i_2,i_3,\cdots, i_m)\neq(i_1,i_1,\cdots,i_1)} (\gamma_{i_1}^+(\mathcal{A})-a_{i_1i_2\cdots i_m})).
  \end{align*}
Since $\mathcal{A}$ is a circulant tensor, then
$$(a_{11\cdots 1}-\gamma_{1}^+(\mathcal{A}))^2>(\sum\limits_{(i_2,i_3,\cdots, i_m)\neq(1,1,\cdots,1)} (\gamma_{1}^+(\mathcal{A})-a_{1i_2\cdots i_m}))^2.$$
According the definitions of double B-tensor and $\gamma_{1}^+(\mathcal{A})$, there is
$$(a_{11\cdots 1}-\gamma_{1}^+(\mathcal{A}))>(\sum\limits_{(i_2,i_3,\cdots, i_m)\neq(1,1,\cdots,1)} (\gamma_{1}^+(\mathcal{A})-a_{1i_2\cdots i_m}))$$

``$(c)\Rightarrow(a)$''. As $\mathcal{A}$ is a circulant tensor, and
$a_{11\cdots 1}-\gamma_{1}^+(\mathcal{A})> \sum\limits_{(i_2,i_3,\cdots,i_m)\neq(1,1,\cdots,1)}(\gamma_{1}^+(\mathcal{A})-a_{1i_2\cdots i_m}),$
then
$a_{i_1i_1\cdots i_1}-\gamma_{i_1}^+(\mathcal{A})> \sum\limits_{(i_2,i_3,\cdots,i_m)\neq(i_1,i_1,\cdots,i_1)}(\gamma_{i_1}^+(\mathcal{A})-a_{i_1i_2\cdots i_m})$
holds for all $i_1\in[n]$. By Proposition \ref{p3}, we can obtain the conclusion.
\end{proof}

Let $\underline{\mathcal{A}}=(\underline{a}_{i_1i_2\cdots i_m})$, $\overline{\mathcal{A}}=(\overline{a}_{i_1i_2\cdots i_m})\in T_{m,n}$ , and $\underline{\mathcal{A}}\leq\overline{\mathcal{A}}$, where $\underline{\mathcal{A}}\leq\overline{\mathcal{A}}$ are to be understood componentwise. The set of tensors
$$\mathcal{A}^I=[\underline{\mathcal{A}},\overline{\mathcal{A}}]= \{\mathcal{A}:\underline{\mathcal{A}}\leq \mathcal{A}\leq \overline{\mathcal{A}}\}$$
is called an interval tensor.

Denote $\mathcal{A}^c=(a^c_{i_1i_2\cdots i_m})=\frac{\underline{\mathcal{A}}+\overline{\mathcal{A}}}{2}$ and $\Delta=(\delta_{i_1i_2\cdots i_m})=\frac{\overline{\mathcal{A}}-\underline{\mathcal{A}}}{2}$, then $\mathcal{A}^I=[\mathcal{A}^c- \Delta, \mathcal{A}^c+ \Delta]$. Clearly, $\Delta$ is a nonnegative tensor, in other words, the values of all its entries $\delta_{i_1i_2\cdots i_m}\geq0$. $\mathcal{A}^I$ is said to be symmetric if both $\mathcal{A}^c$ and $\Delta$ are symmetric.

\begin{definition}\textup{\cite{bf2022,sr2020,boz2020}}
An interval tensor $\mathcal{A}^I$ is said to be an interval B-(double B-, Z-, P-)tensor if each tensor $\mathcal{A}\in\mathcal{A}^I$ is a B-(double B-, Z-, P-)tensor.
\end{definition}


\section{Interval B-tensors}
\setcounter{section}{3}
This section investigates the properties and criteria for interval B-tensors, and establishes equivalent characterizations that depend solely on the extreme point tensors of the interval. The following theorem is the central result of this section, providing a verifiable necessary and sufficient condition for interval B-tensors.
\begin{theorem}\label{th1}
Let $\mathcal{A}^I$ be an interval tensor in $T_{m,n}$. $\mathcal{A}^I$ is an interval B-tensor if and only if
\begin{itemize}
  \item [(a)] $\sum\limits_{i_2,i_3,\cdots,i_m=1}^n\underline{a}_{i_1i_2\cdots i_m}>0$,
  \item [(b)] $\sum\limits_{(i_2,i_3,\cdots,i_m)\neq(j_2,j_3,\cdots,j_m)}\underline{a}_{i_1i_2\cdots i_m}>(n^{m-1}-1)\overline{a}_{i_1j_2\cdots j_m}$ for all $(j_2,j_3,\cdots,j_m)\neq(i_1,i_1,\cdots,i_1)$,
\end{itemize}
are satisfied for each $i_1\in [n]$.
\end{theorem}
\begin{proof} ``{\bf Necessity.}'' As $\underline{\mathcal{A}}\in\mathcal{A}^I$, then $\underline{\mathcal{A}}$ is a B-tensor. According $(a)$ of Definition \ref{def-bt}, $(a)$ of this theorem is held for each $i_1\in [n]$.

For $i_1\in[n]$, and $(j_2,j_3,\cdots,j_m)\neq (i_1,i_1,\cdots,i_1)$, $j_2,j_3,\cdots,j_m\in[n]$.
Define $\mathcal{A}=(a_{i_1i_2\cdots i_m})\in T_{m,n}$ as
$$a_{i_1i_2\cdots i_m}=\begin{cases}
\overline{a}_{i_1i_2\cdots i_m},\mbox{ if }(i_2,i_3,\cdots,i_m)=(j_2,j_3,\cdots,j_m),\\
\underline{a}_{i_1i_2\cdots i_m},\mbox{ otherwise}.
\end{cases}$$
Obviously $\mathcal{A}\in\mathcal{A}^I$, then all such tensors are a B-tensor. According $(b)$ of Definition \ref{def-bt}, $(b)$ of this theorem is held for each $i_1\in [n]$.

``{\bf Sufficiency.}'' Since for any $\mathcal{A}\in\mathcal{A}^I$, there is $\mathcal{A}\geq\underline{\mathcal{A}}$, then $(a)$ of Definition \ref{def-bt} is satisfied for each $\mathcal{A}\in\mathcal{A}^I$ can be deduced by $\sum\limits_{i_2,i_3,\cdots,i_m=1}^n\underline{a}_{i_1i_2\cdots i_m}>0$.

As for any $(j_2,j_3,\cdots,j_m)\neq(i_1,i_1,\cdots,i_1)$,
$$\frac{1}{n^{m-1}}(\sum\limits_{i_2,i_3,\cdots, i_m=1}^n a_{i_1i_2\cdots i_m})>a_{i_1j_2\cdots j_m}\Leftrightarrow \sum\limits_{(i_2,i_3,\cdots,i_m)\neq(j_2,j_3,\cdots,j_m)}a_{i_1i_2\cdots i_m}>(n^{m-1}-1)a_{i_1j_2\cdots j_m},$$
the final inequality contains no repeated entries, therefore, $(b)$ of Definition \ref{def-bt} holds for all $\mathcal{A}\in\mathcal{A}^I$ if and only if the final inequality holds for $\overline{a}_{i_1j_2\cdots j_m}$ on the right side and $\underline{a}_{i_1i_2\cdots i_m}$ on the left side.
\end{proof}

This theorem shows that determining whether an interval tensor is of type B can be fully decided by the row sums of the lower-bound tensor $\underline{\mathcal{A}}$ and specific off-diagonal entries of the upper-bound tensor $\overline{\mathcal{A}}$. We then present several important corollaries of this theorem, which further simplify the criteria or reveal structural properties of interval B-tensors.

\begin{corollary}\label{c1}
Let $\mathcal{A}^I$ and $\mathcal{B}^I$ be two interval tensors in $T_{m,n}$, where $\underline{b}_{i_1i_1\cdots i_1}=\overline{b}_{i_1i_1\cdots i_1}=\underline{a}_{i_1i_1\cdots i_1}$ for all $i_1\in [n]$, and $\underline{b}_{i_1i_2\cdots i_m}=\underline{a}_{i_1i_2\cdots i_m}$ and $\overline{b}_{i_1i_2\cdots i_m}=\overline{a}_{i_1i_2\cdots i_m}$ for all off-diagonal. Then $\mathcal{A}^I$ is an interval B-tensor if and only if $\mathcal{B}^I$ is an interval B-tensor.
\end{corollary}
\begin{proof}
By Theorem \ref{th1}, diagonal entries appear solely as $\underline{a}_{i_1i_1\cdots i_1}$ for all $i_1\in[n]$, so other values of diagonal entries are irrelevant. Hence, $\mathcal{A}^I$ satisfies the conditions $(a)$ and $(b)$ of Theorem \ref{th1} if and only if $\mathcal{B}^I$ satisfies the conditions $(a)$ and $(b)$ of  Theorem \ref{th1}, that is $\mathcal{A}^I$ is an interval B-tensor if and only if $\mathcal{B}^I$ is an interval B-tensor.
\end{proof}

The following corollary offers a way to simplify the verification by comparing the sizes of elements within the interval.
\begin{corollary}\label{c2}
Let $\mathcal{A}^I$ and $\mathcal{B}^I$ be two interval tensors in $T_{m,n}$, and
\begin{align*}K=\{&(i_1,j_2,\cdots,j_m):i_1,j_2,\cdots,j_m\in[n], (j_2,j_3,\cdots,j_m)\neq(i_1,i_1,\cdots,i_1),\\
& \exists (i_2,i_3,\cdots,i_m)\neq(i_1,i_1,\cdots,i_1) \mbox{ and }(i_2,i_3,\cdots,i_m)\neq(j_2,j_3,\cdots,j_m) \\
&\mbox{ s.t. }\overline{a}_{i_1j_2\cdots j_m}\leq \underline{a}_{i_1i_2\cdots i_m}\}.
\end{align*}
Define $\mathcal{B}^I$ as
$\underline{b}_{i_1j_2\cdots j_m}=\overline{b}_{i_1j_2\cdots j_m}=\underline{a}_{i_1j_2\cdots j_m}$ for all $(i_1,j_2,\cdots, j_m)\in K$, and $\underline{b}_{i_1j_2\cdots j_m}=\underline{a}_{i_1j_2\cdots j_m}$ and $\overline{b}_{i_1j_2\cdots j_m}=\overline{a}_{i_1j_2\cdots j_m}$ for otherwise. Then $\mathcal{A}^I$ is an interval B-tensor if and only if $\mathcal{B}^I$ is an interval B-tensor.
\end{corollary}
\begin{proof}
As $\mathcal{B}^I\subset \mathcal{A}^I$, the necessity is evident, we proceed to discuss the sufficiency.

When $K=\emptyset$, $\mathcal{B}^I=\mathcal{A}^I$, is trivial. Suppose $K\neq\emptyset$, compare $\mathcal{B}^I$ and $\mathcal{A}^I$, it is sufficient to show that $(b)$ of Theorem \ref{th1} holds when $(i_1,j_2,\cdots, j_m)\in K$ for $\mathcal{A}^I$.

Take $(i_1,j_2,\cdots, j_m)\in K$, and let
$(i_2,i_3,\cdots,i_m)\neq(i_1,i_1,\cdots,i_1)$ s.t. $\underline{a}_{i_1j_2\cdots j_m}\leq \underline{a}_{i_1i_2\cdots i_m}$ and $\overline{a}_{i_1j_2\cdots j_m}\leq \overline{a}_{i_1i_2\cdots i_m}$.
Then
\begin{align*}\sum\limits_{(k_2,k_3,\cdots,k_m)\neq(j_2,j_3,\cdots,j_m)}\underline{a}_{i_1k_2\cdots k_m}&\geq\sum\limits_{(k_2,k_3,\cdots,k_m)\neq(i_2,i_3,\cdots,i_m)}\underline{a}_{i_1k_2\cdots k_m}=\sum\limits_{(k_2,k_3,\cdots,k_m)\neq(i_2,i_3,\cdots,i_m)}\underline{b}_{i_1k_2\cdots k_m}\\
&>(n^{m-1}-1)\overline{b}_{i_1i_2\cdots i_m}\geq(n^{m-1}-1)\overline{a}_{i_1i_2\cdots i_m}\\
&\geq(n^{m-1}-1)\overline{a}_{i_1j_2\cdots j_m}.
\end{align*}
The second inequality holds as $\mathcal{B}^I$ is an interval B-tensor, which implies that $(b)$ of Theorem \ref{th1} holds for $(i_1,i_2,\cdots,i_m)$.
\end{proof}

\begin{corollary}\label{c3}
Let $\mathcal{A}^I$ be an interval tensor in $T_{m,n}$. $\mathcal{A}^I$ is an interval B-tensor if and only if for all $i_1\in[n]$,
$$\sum\limits_{i_2,i_3,\cdots,i_m=1}^n \underline{a}_{i_1i_2\cdots i_m}>\max\{0,(n^{m-1}-1)\overline{a}_{i_1j_2\cdots j_m}+\underline{a}_{i_1j_2\cdots j_m}\}\mbox{ for all }(j_2,j_3,\cdots, j_m)\neq(i_1,i_1,\cdots, i_1).$$
\end{corollary}
\begin{proof}
``{\bf Necessity.}'' Since $\mathcal{A}^I$ is an interval B-tensor, then by Theorem \ref{th1}, for each $i_1\in[n]$,
$$\sum\limits_{i_2,\cdots,i_m=1}^n \underline{a}_{i_1i_2\cdots i_m}>0,$$
and for any $(j_2,j_3,\cdots, j_m)\neq(i_1,i_1,\cdots, i_1)$
$$\sum\limits_{(i_2,i_3,\cdots,i_m)\neq(j_2,j_3,\cdots, j_m)}\underline{a}_{i_1i_2\cdots i_m}>(n^{m-1}-1)\overline{a}_{i_1j_2\cdots j_m},$$
then
$$\sum\limits_{i_2,i_3,\cdots,i_m=1}^n\underline{a}_{i_1i_2\cdots i_m}>(n^{m-1}-1)\overline{a}_{i_1j_2\cdots j_m}+\underline{a}_{i_1j_2\cdots j_m}.$$

``{\bf Sufficiency.}''
The condition $(a)$ of Theorem \ref{th1} holds evidently. For each $i_1\in[n]$ and $(j_2,j_3,\cdots, j_m)\neq(i_1,i_1,\cdots, i_1)$,
$$\sum\limits_{i_2,i_3,\cdots,i_m=1}^n \underline{a}_{i_1i_2\cdots i_m}>\max\{0,(n^{m-1}-1)\overline{a}_{i_1j_2\cdots j_m}+\underline{a}_{i_1j_2\cdots j_m}\}\geq(n^{m-1}-1)\overline{a}_{i_1j_2\cdots j_m}+\underline{a}_{i_1j_2\cdots j_m},$$
then
$$\sum\limits_{(i_2,i_3,\cdots,i_m)\neq(j_2,j_3,\cdots, j_m)}\underline{a}_{i_1i_2\cdots i_m}>(n^{m-1}-1)\overline{a}_{i_1j_2\cdots j_m},$$
the condition $(b)$ of Theorem \ref{th1} also holds. Therefore $\mathcal{A}^I$ is an interval B-tensor.
\end{proof}

Corollary \ref{c3} consolidates the two conditions of Theorem \ref{th1} into a single compact inequality, making it more convenient for application. The following two corollaries present another equivalent form of interval B-tensors, revealing their intrinsic connection with diagonally dominant structures.

\begin{corollary}\label{c4}
Let $\mathcal{A}^I$ be an interval tensor in $T_{m,n}$. $\mathcal{A}^I$ is an interval B-tensor if and only if for all $i_1\in[n]$,
\begin{itemize}
  \item [(a)] $\sum\limits_{i_2,i_3,\cdots,i_m=1}^n\underline{a}_{i_1i_2\cdots i_m}>0$,
  \item [(b)] $\underline{a}_{i_1i_1\cdots i_1}-\underline{a}_{i_1j_2\cdots j_m}>\sum\limits_{(i_2,i_3,\cdots,i_m)\neq(i_1,i_1,\cdots,i_1)}(\overline{a}_{i_1j_2\cdots j_m}-\underline{a}_{i_1i_2\cdots i_m})$ for all $(j_2,j_3,\cdots,j_m)\neq(i_1,i_1,\cdots,i_1)$.
\end{itemize}
\end{corollary}

\begin{corollary}\label{c5}
Let $\mathcal{A}^I$ be an interval tensor in $T_{m,n}$. $\mathcal{A}^I$ is an interval B-tensor if and only if for all $i_1\in[n]$,
\begin{itemize}
  \item [(a)] $\sum\limits_{i_2,i_3,\cdots,i_m=1}^n\underline{a}_{i_1i_2\cdots i_m}>0$,
  \item [(b)] $\underline{a}_{i_1i_1\cdots i_1}-\overline{a}_{i_1j_2\cdots j_m}>\sum\limits_{\substack{(i_2,i_3,\cdots,i_m)\neq(i_1,i_1,\cdots,i_1)\\
      (i_2,i_3,\cdots,i_m)\neq(j_2,j_3,\cdots,j_m)}}(\overline{a}_{i_1j_2\cdots j_m}-\underline{a}_{i_1i_2\cdots i_m})$ for all $(j_2,j_3,\cdots,j_m)\neq(i_1,i_1,\cdots,i_1)$.
\end{itemize}
\end{corollary}

There exists a significant connection between interval B-tensors and P-tensors, particularly under the conditions of even order and symmetry. This relationship is further elucidated in the corollaries that follow.
\begin{corollary}\label{c61}
Suppose $\mathcal{A}^I $ is an interval Z tensor in $T_{m,n}$ with $m$ is even. Then $\mathcal{A}^I$ is an interval P tensor if $\mathcal{A}^I$ is an interval B-tensor.
\end{corollary}

\begin{corollary}\label{c62}
If $\mathcal{A}^I$ is a symmetric interval B-tensor in $T_{m,n}$ with $m$ is even, then $\mathcal{A}^I$ is an interval P tensor.
\end{corollary}
\begin{proof}
Denote
$$\mathcal{A}^z:=\mathcal{A}^c-\Delta\times_1 T_z\times_2 T_z\times_3\cdots\times_m T_z,$$
where $T_z$ is a $n\times n$ diagonal matrix whose diagonal entries from the vector $z\in Y= \{z\in \mathbb{R}^n; |z_j|=1 \mbox{ for } j=1,\cdots,n\}$.

Let $\mathcal{A}^I$ is a symmetric interval B-tensor, then $\mathcal{A}^z\in\mathcal{A}^I$ is a symmetric B-tensor for all $z\in Y$, hence $\mathcal{A}^z\in\mathcal{A}^I$ is a P tensor for all $z\in Y$ by Proposition \ref{p1}.

For each $\mathcal{A}\in \mathcal{A}^I$, $x=(x_1,x_2,\cdots,x_n)^{\top}\in\mathbb{R}^n$ and $i\in[n]$,
$$|x_i(\mathcal{A}x^{m-1})_i-x_i(\mathcal{A}^cx^{m-1})_i|=|x_i(\mathcal{A}-\mathcal{A}^c)x^{m-1}_i|\leq|x_i|(\Delta |x^{m-1}|)_i.$$
Let $z=\mbox{sgn}(x)=(\mbox{sgn}(x_1),\mbox{sgn}(x_2),\cdots,\mbox{sgn}(x_n))^{\top}$, where
$\mbox{sgn}(x_i)=\begin{cases}
\  \ 1\mbox{ if }x_i\geq0,\\
-1\mbox{ if }x_i<0.
\end{cases}$
Then
\begin{align*}
x_i(\mathcal{A}x^{m-1})_i &\geq x_i(\mathcal{A}^cx^{m-1})_i-|x_i|(\Delta |x^{m-1}|)_i\\
&=\sum\limits_{i_2,\cdots,i_m=1}^na^c_{ii_2\cdots i_m}x_ix_{i_2}\cdots x_{i_m}-\sum\limits_{i_2,\cdots,i_m=1}^n\delta_{ii_2\cdots i_m}|x_ix_{i_2}\cdots x_{i_m}|\\
&=\sum\limits_{i_2,\cdots,i_m=1}^n[a^c_{ii_2\cdots i_m}-\delta_{ii_2\cdots i_m}\mbox{sgn}(x_ix_{i_2}\cdots x_{i_m})]x_ix_{i_2}\cdots x_{i_m}\\
&=x_i(\mathcal{A}^z x^{m-1})_i>0.
\end{align*}
Therefore $\mathcal{A}$ is a P tensor, furthermore, $\mathcal{A}$ is a interval P tensor.
\end{proof}

The proof of this corollary leverages the structure of symmetric interval tensors, connecting any tensor within the interval to a symmetric B-tensor (and hence a P-tensor) through a sign transformation. The following corollary indicates that interval B-tensors possess a certain closure property under row-wise structural mixing.
\begin{corollary}\label{c7}
Suppose $\mathcal{A}^I,\mathcal{B}^I,\mathcal{C}^I$ are interval tensors in $T_{m,n}$, where $\mathcal{A}^I$ and $\mathcal{B}^I$ are interval B-tensors. If for each $i_1\in[n]$, either
$\underline{c}_{i_1i_1\cdots i_1}=\underline{a}_{i_1i_1\cdots i_1}$, $\overline{c}_{i_1i_1\cdots i_1}=\overline{a}_{i_1i_1\cdots i_1}$, the off-diagonal entries of the $i_1$-th row in $\underline{\mathcal{A}}$ and $\underline{\mathcal{C}}$ forms bijection and the off-diagonal entries of the $i_1$-th row in $\overline{\mathcal{A}}$ and $\overline{\mathcal{C}}$ forms bijection,
or $\underline{c}_{i_1i_1\cdots i_1}=\underline{b}_{i_1i_1\cdots i_1}$, $\overline{c}_{i_1i_1\cdots i_1}=\overline{b}_{i_1i_1\cdots i_1}$, the off-diagonal entries of the $i_1$-th row in $\underline{\mathcal{B}}$ and $\underline{\mathcal{C}}$ forms bijection and the off-diagonal entries of the $i_1$-th row in $\overline{\mathcal{B}}$ and $\overline{\mathcal{C}}$ forms bijection. then $\mathcal{C}^I$ is also an interval B-tensor.
\end{corollary}

\begin{remark}\label{r2}
Since for an interval B-tensor $\mathcal{A}^I$ in $T_{m,n}$, every $\mathcal{A}\in\mathcal{A}^I$ is a B-tensor. Thus define $\mathcal{A}'=(a'_{i_1i_2\cdots i_m})$ as
$$a'_{i_1i_2\cdots i_m}=\begin{cases}
\underline{a}_{i_1i_2\cdots i_m},\mbox{ if }i_1=i_2=\cdots=i_m,\\
\overline{a}_{i_1i_2\cdots i_m},\mbox{ otherwise},
\end{cases}$$
is a B-tensor, then for all $i_1\in[n]$
$$\underline{a}_{i_1i_1\cdots i_1}>\max\{0,\overline{a}_{i_1i_2\cdots i_m}:(i_2,i_3,\cdots, i_m)\neq(i_1,i_1,\cdots,i),i_2,i_3,\cdots,i_m\in[n]\}.$$
\end{remark}

For the special and important subclass of interval Z tensors, the determination of the property of being an interval B-tensor can be greatly simplified, as shown in the following proposition.

\begin{proposition}\label{p6}
If $\mathcal{A}^I$ is an interval Z tensor in $T_{m,n}$, then the following are equivalent.
\begin{itemize}
  \item [(a)] $\mathcal{A}^I$ is an interval B-tensor.
  \item [(b)] $\sum\limits_{i_2,i_3,\cdots, i_m=1}^n \underline{a}_{i_1i_2\cdots i_m}>0$ for each $i_1\in[n]$.
  \item [(c)] $\underline{\mathcal{A}}$ is strictly diagonally dominant.
  \item [(d)] $\underline{\mathcal{A}}$ is a B-tensor.
\end{itemize}
\end{proposition}
\begin{proof}
``$(a)\Rightarrow(b)$'' can be deduced by Theorem \ref{th1}, and the equivalence of $(b)$, $(c)$ and $(d)$ is the result of Lemma \ref{p4}. We show ``$(c)\Rightarrow(a)$'' in the following.

Any $\mathcal{A}\in\mathcal{A}^I$, for all $i_1\in[n]$, it satisfies that
$$a_{i_1i_1\cdots i_1}\geq \underline{a}_{i_1i_1\cdots i_1}\mbox{ and } |a_{i_1i_2\cdots i_m}|\leq|\underline{a}_{i_1i_2\cdots i_m}|\mbox{ for all }(i_2,i_3,\cdots,i_m)\neq(i_1,i_1,\cdots,i_1).$$
Then
$$a_{i_1i_1\cdots i_1}\geq \underline{a}_{i_1i_1\cdots i_1}>\sum\limits_{(i_2,i_3,\cdots, i_m)\neq(i_1,i_1,\cdots,i_1)} |\underline{a}_{i_1i_2\cdots i_m}|\geq\sum\limits_{(i_2,i_3,\cdots, i_m)\neq(i_1,i_1,\cdots,i_1)} |a_{i_1i_2\cdots i_m}|\geq0.$$
Hence $\mathcal{A}$ is a strictly diagonally dominant tensor. Then according Lemma \ref{p4}, $\mathcal{A}$ is a B-tensor, therefore $\mathcal{A}^I$ is an interval B-tensor.
\end{proof}

\begin{proposition}\label{p7}
If $\mathcal{A}^I$ is an interval B-tensor in $T_{m,n}$. Then for all $i_1\in[n]$ the following are satisfied.
\begin{itemize}
  \item [(a)] $\underline{a}_{i_1i_1\cdots i_1}>\sum\limits_{(i_2,i_3,\cdots, i_m)\in J} |\underline{a}_{i_1i_2\cdots i_m}|$, where $J=\{(i_2,i_3,\cdots, i_m):i_2,i_3,\cdots, i_m\in[n],\underline{a}_{i_1i_2\cdots i_m}<0\}$.
  \item [(b)] $\underline{a}_{i_1i_1\cdots i_1}>\max\{|\overline{a}_{i_1i_2\cdots i_m}|,|\underline{a}_{i_1i_2\cdots i_m}|\}$ for all $(i_2,i_3,\cdots, i_m)\neq (i_1,i_1,\cdots,i_1)$.
\end{itemize}
\end{proposition}
\begin{proof}
$(a)$ For all $i\in[n]$, we can divide it into the following two cases.

$(1)$ $\underline{a}_{i_1i_2\cdots i_m}\leq0$ for all $(i_2,i_3,\cdots, i_m)\neq (i_1,i_1,\cdots,i_1)$. By $(a)$ of Theorem \ref{th1}, the conclusion can be directly derived.

$(2)$ $\underline{a}_{i_1i_2\cdots i_m}>0$ for some $(i_2,i_3,\cdots, i_m)\neq (i_1,i_1,\cdots,i_1)$. Denote $\underline{a}_{i_1j_2\cdots j_m}\in\max\{\underline{a}_{i_1i_2\cdots i_m}:(i_2,i_3,\cdots, i_m)\neq (i_1,i_1,\cdots,i_1)\}$. By Corollary \ref{c4},
$$\underline{a}_{i_1i_1\cdots i_1}-\underline{a}_{i_1j_2\cdots j_m}>\sum\limits_{(i_2,i_3,\cdots, i_m)\neq (i_1,i_1,\cdots,i_1)}(\overline{a}_{i_1j_2\cdots j_m}-\underline{a}_{i_1i_2\cdots i_m}).$$
Since
$$\underline{a}_{i_1i_1\cdots i_1}>\underline{a}_{i_1i_1\cdots i_1}-\underline{a}_{i_1j_2\cdots j_m}$$
and
$$\overline{a}_{i_1j_2\cdots j_m}-\underline{a}_{i_1i_2\cdots i_m}\geq0 \mbox{ for all }(i_2,i_3,\cdots, i_m)\neq (i_1,i_1,\cdots,i_1),$$
then
\begin{align*}\underline{a}_{i_1i_1\cdots i_1}&>\underline{a}_{i_1i_1\cdots i_1}-\underline{a}_{i_1j_2\cdots j_m}\\
&>\sum\limits_{(i_2,i_3,\cdots, i_m)\neq (i_1,i_1,\cdots,i_1)}(\overline{a}_{i_1j_2\cdots j_m}-\underline{a}_{i_1i_2\cdots i_m})\\
&\geq\sum\limits_{(i_2,i_3,\cdots, i_m)\in J}(\overline{a}_{i_1j_2\cdots j_m}-\underline{a}_{i_1i_2\cdots i_m})\\
&>\sum\limits_{(i_2,i_3,\cdots, i_m)\in J}(-\underline{a}_{i_1i_2\cdots i_m})\\
&=\sum\limits_{(i_2,i_3,\cdots, i_m)\in J}|\underline{a}_{i_1i_2\cdots i_m}|.
\end{align*}

$(b)$ For all $(i_2,i_3,\cdots, i_m)\neq (i_1,i_1,\cdots,i_1)$, we can divide it into the following there cases.

$(1)$ When $|\underline{a}_{i_1i_2\cdots i_m}|>|\overline{a}_{i_1i_2\cdots i_m}|$, then $\underline{a}_{i_1i_2\cdots i_m}\leq 0$, this implies $(i_2,i_3,\cdots, i_m)\in J$. According $(a)$,
$$\underline{a}_{i_1i_1\cdots i_1}>\sum\limits_{(j_2,j_3,\cdots, j_m)\in J} |\underline{a}_{i_1j_2\cdots j_m}|\geq |\underline{a}_{i_1i_2\cdots i_m}|.$$

$(2)$ When $|\overline{a}_{i_1i_2\cdots i_m}|>|\underline{a}_{i_1i_2\cdots i_m}|$, then $\overline{a}_{i_1i_2\cdots i_m}>0$.  According Remark \ref{r2},
$$\underline{a}_{i_1i_1\cdots i_1}>\overline{a}_{i_1i_2\cdots i_m}=|\overline{a}_{i_1i_2\cdots i_m}|.$$

$(3)$ When $|\overline{a}_{i_1i_2\cdots i_m}|=|\underline{a}_{i_1i_2\cdots i_m}|$, then either $\overline{a}_{i_1i_2\cdots i_m}=-\underline{a}_{i_1i_2\cdots i_m}$, or $\overline{a}_{i_1i_2\cdots i_m}=\underline{a}_{i_1i_2\cdots i_m}$, this implies $\underline{a}_{i_1i_2\cdots i_m}<0$ or $\overline{a}_{i_1i_2\cdots i_m}>0$. Hence, we can draw the conclusion based on the discussions of the two previous cases.
\end{proof}

\begin{example}
Consider an interval tensor $\mathcal{A}^I = [\underline{\mathcal{A}}, \overline{\mathcal{A}}] \in T_{3,2}$ where
$\underline{a}_{111}=\underline{a}_{222}=4$, $\underline{a}_{112}=\underline{a}_{121}=\underline{a}_{211}=0$, $\underline{a}_{122}=\underline{a}_{212}= \underline{a}_{221}=1$, $\overline{a}_{111}=\overline{a}_{222}=5$, $\overline{a}_{112}=\overline{a}_{121}=\overline{a}_{211}=1$, $\overline{a}_{122}=\overline{a}_{212}=\overline{a}_{221}=2$.
\end{example}
We check the conditions of Theorem \ref{th1} for $i_1 = 1$.
$\sum\limits_{i_2,i_3=1}^2 \underline{a}_{1i_2i_3} = 4+0+0+1 = 5 > 0$, then ($a$) in Theorem \ref{th1} is satisfied for $i_1=1$.
But when $(j_2,j_3) = (2,2) \neq (1,1)$, there are
$$\sum\limits_{(i_2,i_3)\neq(2,2)} \underline{a}_{1i_2i_3} = 4+0+0 = 4,$$
and
$$(2^{2}-1)\overline{a}_{122} = 3 \times 2 = 6,$$
($b$) in Theorem \ref{th1} is not satisfied. Therefore, $\mathcal{A}^I$ is not an interval B-tensor.

However, if we tighten the upper bound $\overline{a}_{122} \leq 1$, then $4 > 3\times1 = 3$ holds, and similarly check other indices.

\vspace{.3cm}

\section{Interval Double B-tensors}
\setcounter{section}{4}
This section extends the study to interval double B-tensors. We first note that, under the conditions of even order and symmetry, interval double B-tensors also ensure the property of being an interval P-tensor.
\begin{corollary}\label{c8}
If $\mathcal{A}^I$ is a symmetric interval double B-tensor in $T_{m,n}$ with $m$ is even, then $\mathcal{A}^I$ is an interval P tensor.
\end{corollary}
The proof of Corollary \ref{c8} is analogous to that of Corollary \ref{c62} and thus omitted herein.

The determination of interval double B-tensors is more complex than that of interval B-tensors, yet their necessary and sufficient conditions can still be expressed as a set of inequalities dependent solely on the interval endpoints. The following theorem is the primary result of this section.
\begin{theorem}\label{th2}
If $\mathcal{A}^I$ is an interval tensor in $T_{m,n}$. Then $\mathcal{A}^I$ is an interval double B-tensor if and only if
\begin{itemize}
  \item [(a)] For each $i_1\in [n]$, $\underline{a}_{i_1i_1\cdots i_1}>\max\{0,\overline{a}_{i_1i_2\cdots i_m}:(i_2,i_3,\cdots,i_m)\neq(i_1,i_1,\cdots,i_1)\}$;
  \item [(b)] For all $i_1\in [n]$, and $(i_2,i_3,\cdots,i_m)\neq (i_1,i_1,\cdots,i_1)$:
  \begin{itemize}
  \item [(b$_1$)] $\underline{a}_{i_1i_1\cdots i_1}-\overline{a}_{i_1i_2\cdots i_m}\geq
      \max\left\{0,\sum\limits_{\substack{(k_2,k_3,\cdots,k_m)\neq(i_1,i_1,\cdots,i_1)\\
      (k_2,k_3,\cdots,k_m)\neq(i_2,i_3,\cdots,i_m)}}(\overline{a}_{i_1i_2\cdots i_m}-\underline{a}_{i_1k_2\cdots k_m})\right\}$;
  \item [(b$_2$)] $\underline{a}_{i_1i_1\cdots i_1}\geq\left\{0,-\sum\limits_{(k_2,k_3,\cdots,k_m)\neq(i_1,i_1,\cdots,i_1)}\underline{a}_{i_1k_2\cdots k_m}\right\}$;
  \end{itemize}
  \item [(c)] For all $i_1,j_1\in [n]$, $i_1\neq j_1$, and $(i_2,i_3,\cdots,i_m)\neq (i_1,i_1,\cdots,i_1)$, $(j_2,j_3,\cdots,j_m)\neq (j_1,j_1,\cdots,j_1)$:
  \begin{itemize}
  \item [(c$_1$)] $(\underline{a}_{i_1i_1\cdots i_1}-\overline{a}_{i_1i_2\cdots i_m})\cdot(\underline{a}_{j_1j_1\cdots j_1}-\overline{a}_{j_1j_2\cdots j_m})>\\
      \left(\max\left\{0,\sum\limits_{\substack{(k_2,k_3,\cdots,k_m)\neq(i_1,i_1,\cdots,i_1)\\
      (k_2,k_3,\cdots,k_m)\neq(i_2,i_3,\cdots,i_m)}}(\overline{a}_{i_1i_2\cdots i_m}-\underline{a}_{i_1k_2\cdots k_m})\right\}\right)\cdot\\
      \left(\max\left\{0,\sum\limits_{\substack{(k_2,k_3,\cdots,k_m)\neq(j_1,j_1,\cdots,j_1)\\
      (k_2,k_3,\cdots,k_m)\neq(j_2,j_3,\cdots,j_m)}}(\overline{a}_{j_1j_2\cdots j_m}-\underline{a}_{j_1k_2\cdots k_m})\right\}\right)$;
  \item [(c$_2$)] $(\underline{a}_{i_1i_1\cdots i_1}-\overline{a}_{i_1i_2\cdots i_m})\cdot\underline{a}_{j_1j_1\cdots j_1}>\left(\max\left\{0,\sum\limits_{\substack{(k_2,k_3,\cdots,k_m)\neq(i_1,i_1,\cdots,i_1)\\
      (k_2,k_3,\cdots,k_m)\neq(i_2,i_3,\cdots,i_m)}}(\overline{a}_{i_1i_2\cdots i_m}-\underline{a}_{i_1k_2\cdots k_m})\right\}\right)\cdot\\
      \left(\max\left\{0,-\sum\limits_{(k_2,k_3,\cdots,k_m)\neq(j_1,j_1,\cdots,j_1)}\underline{a}_{j_1k_2\cdots k_m}\right\}\right)$;
  \item [(c$_3$)]$\underline{a}_{i_1i_1\cdots i_1}\cdot\underline{a}_{j_1j_1\cdots j_1}>\left(\max\left\{0,-\sum\limits_{(k_2,k_3,\cdots,k_m)\neq(i_1,i_1,\cdots,i_1)}\underline{a}_{i_1k_2\cdots k_m}\right\}\right)\cdot\\
      \left(\max\left\{0,-\sum\limits_{(k_2,k_3,\cdots,k_m)\neq(j_1,j_1,\cdots,j_1)}\underline{a}_{j_1k_2\cdots k_m}\right\}\right)$.
      \end{itemize}
\end{itemize}
\end{theorem}
\begin{proof}
``{\bf Necessity.}'' Since $\mathcal{A}^I$ is an interval double B-tensor, then all $\mathcal{A}\in\mathcal{A}^I$ is a double B-tensor.

$(a)$ The tensor $\mathcal{A}'$ defined in Remark \ref{r2} is a double B-tensor as $\mathcal{A}'\in\mathcal{A}^I$. According to $(a)$ of Definition \ref{def-dbt},
for each $i_1\in [n]$,
\begin{align*}
\underline{a}_{i_1i_1\cdots i_1}&=a'_{i_1i_1\cdots i_1}\\
&>\max\{0,a'_{i_1i_2\cdots i_m}:(i_2,i_3,\cdots,i_m)\neq(i_1,i_1,\cdots,i_1)\}\\
&=\max\{0,\overline{a}_{i_1i_2\cdots i_m}:(i_2,i_3,\cdots,i_m)\neq(i_1,i_1,\cdots,i_1)\}.
\end{align*}

$(b)$  For any $i_1,j_1\in [n]$, $i_1\neq j_1$, and $(i_2,i_3,\cdots,i_m)\neq (i_1,i_1,\cdots,i_1)$, $(j_2,j_3,\cdots,j_m)\neq (j_1,j_1,\cdots,j_1)$:

$(b_1)$ Define $\mathcal{A}^1=(a^1_{l_1l_2\cdots l_m})$ as
$$a^1_{l_1l_2\cdots l_m}=\begin{cases}
\overline{a}_{i_1i_2\cdots i_m},\mbox{ if }(l_1,l_2,\cdots,i_m)=(i_1,i_2,\cdots,i_m),\\
\underline{a}_{l_1l_2\cdots l_m}\mbox{ otherwise}.
\end{cases}$$
Clearly $\mathcal{A}^1\in\mathcal{A}^I$ is a double B-tensor. Then
\begin{align*}
\underline{a}_{i_1i_1\cdots i_1}-\overline{a}_{i_1i_2\cdots i_m} &\geq a^1_{i_1i_1\cdots i_1}-\gamma_{i_1}^+(\mathcal{A}^1)\\
&\geq\sum\limits_{(k_2,k_3,\cdots,k_m)\neq(i_1,i_1,\cdots,i_1)}(\gamma_{i_1}^+(\mathcal{A}^1)-a^1_{i_1k_2\cdots k_m})\\
&\geq\max\left\{0,\sum\limits_{\substack{(k_2,k_3,\cdots,k_m)\neq(i_1,i_1,\cdots,i_1)\\
      (k_2,k_3,\cdots,k_m)\neq(i_2,i_3,\cdots,i_m)}}(\overline{a}_{i_1i_2\cdots i_m}-\underline{a}_{i_1k_2\cdots k_m})\right\}.
\end{align*}
The second inequality holds as $\mathcal{A}^1$ is a double B-tensor, and last inequality holds as $\gamma_{i_1}^+(\mathcal{A}^1)\geq a^1_{i_1k_2\cdots k_m}$ for all $(k_2,k_3,\cdots,k_m)\neq(i_1,i_1,\cdots,i_1)$, and $\gamma_{i_1}^+(\mathcal{A}^1)\geq a^1_{i_1i_2\cdots i_m}=\overline{a}_{i_1i_2\cdots i_m}$.

$(b_2)$ Since $\underline{\mathcal{A}}\in\mathcal{A}^I$ is a double B-tensor. Then
\begin{align*}
\underline{a}_{i_1i_1\cdots i_1}&\geq\underline{a}_{i_1i_1\cdots i_1}-\gamma_{i_1}^+(\underline{\mathcal{A}})\\
&\geq\sum\limits_{(k_2,k_3,\cdots,k_m)\neq(i_1,i_1,\cdots,i_1)}(\gamma_{i_1}^+(\underline{\mathcal{A}})-\underline{a}_{i_1k_2\cdots k_m})\\
&\geq\max\left\{0,-\sum\limits_{(k_2,k_3,\cdots,k_m)\neq(i_1,i_1,\cdots,i_1)}\underline{a}_{i_1k_2\cdots k_m}\right\}.
\end{align*}
These inequalities hold due to $\gamma_{i_1}^+(\underline{\mathcal{A}})\geq 0$ and reasoning similar to that the proof of $(b_2)$.

$(c)$ Similar to the proof of $(b)$, for any $i_1,j_1\in [n]$, $i_1\neq j_1$, and $(i_2,i_3,\cdots,i_m)\neq (i_1,i_1,\cdots,i_1)$, $(j_2,j_3,\cdots,j_m)\neq (j_1,j_1,\cdots,j_1)$:

$(c_1)$ Define $\mathcal{A}^2=(a^2_{l_1l_2\cdots l_m})$ as
$$a^2_{l_1l_2\cdots l_m}=\begin{cases}
\overline{a}_{i_1i_2\cdots i_m},\mbox{ if }(l_1,l_2,\cdots,i_m)=(i_1,i_2,\cdots,i_m),\\
\overline{a}_{j_1j_2\cdots j_m},\mbox{ if }(l_1,l_2,\cdots,i_m)=(j_1,j_2,\cdots,j_m),\\
\underline{a}_{l_1l_2\cdots l_m}\mbox{ otherwise}.
\end{cases}$$
Clearly $\mathcal{A}^2\in\mathcal{A}^I$ is a double B-tensor. Then
\begin{align*}
&(\underline{a}_{i_1i_1\cdots i_1}-\overline{a}_{i_1i_2\cdots i_m})\cdot(\underline{a}_{j_1j_1\cdots j_1}-
\overline{a}_{j_1j_2\cdots j_m})\geq(a^2_{i_1i_1\cdots i_1}-\gamma_{i_1}^+(\mathcal{A}^2))\cdot(a^2_{j_1j_1\cdots j_1}-\gamma_{j_1}^+(\mathcal{A}^2))\\
&>\left(\sum\limits_{(k_2,k_3,\cdots,k_m)\neq(i_1,i_1,\cdots,i_1)}(\gamma_{i_1}^+(\mathcal{A}^2)-a^2_{i_1k_2\cdots k_m})\right)\cdot\left(\sum\limits_{(k_2,k_3,\cdots,k_m)\neq(j_1,j_1,\cdots,j_1)}(\gamma_{j_1}^+(\mathcal{A}^2)-a^2_{j_1k_2\cdots k_m})\right)\\
&\geq\left(\max\left\{0,\sum\limits_{\substack{(k_2,k_3,\cdots,k_m)\neq(i_1,i_1,\cdots,i_1)\\
      (k_2,k_3,\cdots,k_m)\neq(i_2,i_3,\cdots,i_m)}}(\overline{a}_{i_1i_2\cdots i_m}-\underline{a}_{i_1k_2\cdots k_m})\right\}\right)\cdot\\
&~~~\left(\max\left\{0,\sum\limits_{\substack{(k_2,k_3,\cdots,k_m)\neq(j_1,j_1,\cdots,j_1)\\
      (k_2,k_3,\cdots,k_m)\neq(j_2,j_3,\cdots,j_m)}}(\overline{a}_{j_1j_2\cdots j_m}-\underline{a}_{j_1k_2\cdots k_m})\right\}\right).
\end{align*}

$(c_2)$ Since $\mathcal{A}^1\in\mathcal{A}^I$ is a double B-tensor. Then
\begin{align*}
&(\underline{a}_{i_1i_1\cdots i_1}-\overline{a}_{i_1i_2\cdots i_m})\cdot\underline{a}_{j_1j_1\cdots j_1}\geq(a^1_{i_1i_1\cdots i_1}-\gamma_{i_1}^+(\mathcal{A}^1))\cdot(a^1_{j_1j_1\cdots j_1}-\gamma_{j_1}^+(\mathcal{A}^1))\\
&>\left(\sum\limits_{(k_2,k_3,\cdots,k_m)\neq(i_1,i_1,\cdots,i_1)}(\gamma_{i_1}^+(\mathcal{A}^1)-a^1_{i_1k_2\cdots k_m})\right)\cdot\left(\sum\limits_{(k_2,k_3,\cdots,k_m)\neq(j_1,j_1,\cdots,j_1)}(\gamma_{j_1}^+(\mathcal{A}^1)-a^1_{j_1k_2\cdots k_m})\right)\\
&\geq\left(\max\left\{0,\sum\limits_{\substack{(k_2,k_3,\cdots,k_m)\neq(i_1,i_1,\cdots,i_1)\\
      (k_2,k_3,\cdots,k_m)\neq(i_2,i_3,\cdots,i_m)}}(\overline{a}_{i_1i_2\cdots i_m}-\underline{a}_{i_1k_2\cdots k_m})\right\}\right)\cdot \\
& ~~~ \left(\max\left\{0,-\sum\limits_{(k_2,k_3,\cdots,k_m)\neq(j_1,j_1,\cdots,j_1)}\underline{a}_{j_1k_2\cdots k_m})\right\}\right).
\end{align*}

$(c_3)$ Since $\underline{\mathcal{A}}\in\mathcal{A}^I$ is a double B-tensor. Then
\begin{align*}
&\underline{a}_{i_1i_1\cdots i_1}\cdot\underline{a}_{j_1j_1\cdots j_1}\geq(\underline{a}_{i_1i_1\cdots i_1}-\gamma_{i_1}^+(\underline{\mathcal{A}}))\cdot(\underline{a}_{j_1j_1\cdots j_1}-\gamma_{j_1}^+(\underline{\mathcal{A}}))\\
&>\left(\sum\limits_{(k_2,k_3,\cdots,k_m)\neq(i_1,i_1,\cdots,i_1)}(\gamma_{i_1}^+(\underline{\mathcal{A}})-\underline{a}_{i_1k_2\cdots k_m})\right)\cdot\left(\sum\limits_{(k_2,k_3,\cdots,k_m)\neq(j_1,j_1,\cdots,j_1)}(\gamma_{j_1}^+(\underline{\mathcal{A}})-\underline{a}_{j_1k_2\cdots k_m})\right)\\
&\geq\left(\max\left\{0,-\sum\limits_{(k_2,k_3,\cdots,k_m)\neq(i_1,i_1,\cdots,i_1)}\underline{a}_{i_1k_2\cdots k_m})\right\}\right)\cdot\left(\max\left\{0,-\sum\limits_{(k_2,k_3,\cdots,k_m)\neq(j_1,j_1,\cdots,j_1)}(\underline{a}_{j_1k_2\cdots k_m})\right\}\right).
\end{align*}

``{\bf Sufficiency.}'' Let $\mathcal{A}\in\mathcal{A}^I$.

According to $(a)$, for all $i_1\in[n]$,
\begin{align*}
a_{i_1i_1\cdots i_1}&\geq\underline{a}_{i_1i_1\cdots i_1}\\
&>\max\{0,\overline{a}_{i_1i_2\cdots i_m}:(i_2,i_3,\cdots,i_m)\neq(i_1,i_1,\cdots,i_1)\}\\
&\geq\max\{0,a_{i_1i_2\cdots i_m}:(i_2,i_3,\cdots,i_m)\neq(i_1,i_1,\cdots,i_1)\},
\end{align*}
then condition $(a)$ of Definition \ref{def-dbt} is obtained.

To proof condition $(b)$ of Definition \ref{def-dbt}, we divide the discussion into the following two cases for any $i_1\in[n]$.

\textbf{Case $b_1$.} When $\gamma_{i_1}^+(\mathcal{A})>0$, then there is $(i_2,i_3,\cdots,i_m)\neq(i_1,i_1,\cdots,i_1)$ such that $\gamma_{i_1}^+(\mathcal{A})=a_{i_1i_2\cdots i_m}$. So that,
\begin{align*}
a_{i_1i_1\cdots i_1}-\gamma_{i_1}^+(\mathcal{A})&\geq \underline{a}_{i_1i_1\cdots i_1}-\overline{a}_{i_1i_2\cdots i_m}\\
&\geq\max\left\{0,\sum\limits_{\substack{(k_2,k_3,\cdots,k_m)\neq(i_1,i_1,\cdots,i_1)\\
      (k_2,k_3,\cdots,k_m)\neq(i_2,i_3,\cdots,i_m)}}(\overline{a}_{i_1i_2\cdots i_m}-\underline{a}_{i_1k_2\cdots k_m})\right\}\\
&=\sum\limits_{\substack{(k_2,k_3,\cdots,k_m)\neq(i_1,i_1,\cdots,i_1)\\
      (k_2,k_3,\cdots,k_m)\neq(i_2,i_3,\cdots,i_m)}}(\overline{a}_{i_1i_2\cdots i_m}-\underline{a}_{i_1k_2\cdots k_m})\\
&\geq \sum\limits_{(k_2,k_3,\cdots,k_m)\neq(i_1,i_1,\cdots,i_1)}(\gamma_{i_1}^+(\mathcal{A})-a_{i_1k_2\cdots k_m}).
\end{align*}
The second inequality holds due to the condition $(b_1)$.

\textbf{Case $b_2$.} When $\gamma_{i_1}^+(\mathcal{A})=0$, then,
\begin{align*}
a_{i_1i_1\cdots i_1}-\gamma_{i_1}^+(\mathcal{A})&=a_{i_1i_1\cdots i_1}\geq \underline{a}_{i_1i_1\cdots i_1}\\
&\geq\max\left\{0,-\sum\limits_{(k_2,k_3,\cdots,k_m)\neq(i_1,i_1,\cdots,i_1)}\underline{a}_{i_1k_2\cdots k_m}\right\}\\
&=-\sum\limits_{(k_2,k_3,\cdots,k_m)\neq(i_1,i_1,\cdots,i_1)}\underline{a}_{i_1k_2\cdots k_m}\\
&\geq \sum\limits_{(k_2,k_3,\cdots,k_m)\neq(i_1,i_1,\cdots,i_1)}(\gamma_{i_1}^+(\mathcal{A})-a_{i_1k_2\cdots k_m}).
\end{align*}
The second inequality holds due to the condition $(b_2)$.

Then condition $(b)$ of Definition \ref{def-dbt} is obtained.

To proof of condition $(c)$ in Definition \ref{def-dbt} is similar to that of condition $(b)$, we divide the discussion into the following there cases for any $i_1,j_1\in[n]$, $i_1\neq j_1$.

\textbf{Case $c_1$.} When $\gamma_{i_1}^+(\mathcal{A})>0$ and $\gamma_{j_1}^+(\mathcal{A})>0$, then there are $(i_2,i_3,\cdots,i_m)\neq(i_1,i_1,\cdots,i_1)$ and $(j_2,j_3,\cdots,j_m)\neq(j_1,j_1,\cdots,j_1)$ such that $\gamma_{i_1}^+(\mathcal{A})=a_{i_1i_2\cdots i_m}$ and $\gamma_{j_1}^+(\mathcal{A})=a_{j_1j_2\cdots j_m}$. So that,
\begin{align*}
&(a_{i_1i_1\cdots i_1}-\gamma_{i_1}^+(\mathcal{A}))\cdot(a_{j_1j_1\cdots j_1}-\gamma_{j_1}^+(\mathcal{A}))\geq (\underline{a}_{i_1i_1\cdots i_1}-\overline{a}_{i_1i_2\cdots i_m})\cdot(\underline{a}_{j_1j_1\cdots j_1}-\overline{a}_{j_1j_2\cdots j_m})\\
&>\left(\max\left\{0,\sum\limits_{\substack{(k_2,k_3,\cdots,k_m)\neq(i_1,i_1,\cdots,i_1)\\
      (k_2,k_3,\cdots,k_m)\neq(i_2,i_3,\cdots,i_m)}}(\overline{a}_{i_1i_2\cdots i_m}-\underline{a}_{i_1k_2\cdots k_m})\right\}\right)\cdot\\
&~~~\left(\max\left\{0,\sum\limits_{\substack{(k_2,k_3,\cdots,k_m)\neq(j_1,j_1,\cdots,j_1)\\
      (k_2,k_3,\cdots,k_m)\neq(j_2,j_3,\cdots,j_m)}}(\overline{a}_{j_1j_2\cdots j_m}-\underline{a}_{j_1k_2\cdots k_m})\right\}\right)\\
&=\left(\sum\limits_{\substack{(k_2,k_3,\cdots,k_m)\neq(i_1,i_1,\cdots,i_1)\\
      (k_2,k_3,\cdots,k_m)\neq(i_2,i_3,\cdots,i_m)}}(\overline{a}_{i_1i_2\cdots i_m}-\underline{a}_{i_1k_2\cdots k_m})\right)\cdot\left(\sum\limits_{\substack{(k_2,k_3,\cdots,k_m)\neq(j_1,j_1,\cdots,j_1)\\
      (k_2,k_3,\cdots,k_m)\neq(j_2,j_3,\cdots,j_m)}}(\overline{a}_{j_1j_2\cdots j_m}-\underline{a}_{j_1k_2\cdots k_m})\right)\\
&\geq \left(\sum\limits_{(k_2,k_3,\cdots,k_m)\neq(i_1,i_1,\cdots,i_1)}(\gamma_{i_1}^+(\mathcal{A})-a_{i_1k_2\cdots k_m})\right)\cdot\left(\sum\limits_{(k_2,k_3,\cdots,k_m)\neq(j_1,j_1,\cdots,j_1)}(\gamma_{j_1}^+(\mathcal{A})-a_{j_1k_2\cdots k_m})\right).
\end{align*}
The second inequality holds due to the condition $(c_1)$.

\textbf{Case $c_2$.} When one of $\{\gamma_{i_1}^+(\mathcal{A}),\gamma_{j_1}^+(\mathcal{A})\}$ is positive and the other one is zero, without loss the generality, let $\gamma_{i_1}^+(\mathcal{A})>0$ and $\gamma_{j_1}^+(\mathcal{A})=0$, then there is $(i_2,i_3,\cdots,i_m)\neq(i_1,i_1,\cdots,i_1)$ such that $\gamma_{i_1}^+(\mathcal{A})=a_{i_1j_2\cdots j_m}$. So that,
\begin{align*}
&(a_{i_1i_1\cdots i_1}-\gamma_{i_1}^+(\mathcal{A}))\cdot(a_{j_1j_1\cdots j_1}-\gamma_{j_1}^+(\mathcal{A}))=(a_{i_1i_1\cdots i_1}-\gamma_{i_1}^+(\mathcal{A}))\cdot a_{j_1j_1\cdots j_1}\geq (\underline{a}_{i_1i_1\cdots i_1}-\overline{a}_{i_1i_2\cdots i_m})\cdot\underline{a}_{j_1j_1\cdots j_1}\\
&>\left(\max\left\{0,\sum\limits_{\substack{(k_2,k_3,\cdots,k_m)\neq(i_1,i_1,\cdots,i_1)\\
      (k_2,k_3,\cdots,k_m)\neq(i_2,i_3,\cdots,i_m)}}(\overline{a}_{i_1i_2\cdots i_m}-\underline{a}_{i_1k_2\cdots k_m})\right\}\right)\cdot\\
&~~~\left(\max\left\{0,-\sum\limits_{(k_2,k_3,\cdots,k_m)\neq(j_1,j_1,\cdots,j_1)}\underline{a}_{j_1k_2\cdots k_m}\right\}\right)\\
&=\left(\sum\limits_{\substack{(k_2,k_3,\cdots,k_m)\neq(i_1,i_1,\cdots,i_1)\\
      (k_2,k_3,\cdots,k_m)\neq(i_2,i_3,\cdots,i_m)}}(\overline{a}_{i_1i_2\cdots i_m}-\underline{a}_{i_1k_2\cdots k_m})\right)\cdot\left(-\sum\limits_{(k_2,k_3,\cdots,k_m)\neq(j_1,j_1,\cdots,j_1)}\underline{a}_{j_1k_2\cdots k_m}\right)\\
&\geq \left(\sum\limits_{(k_2,k_3,\cdots,k_m)\neq(i_1,i_1,\cdots,i_1)}(\gamma_{i_1}^+(\mathcal{A})-a_{i_1k_2\cdots k_m})\right)\cdot\left(\sum\limits_{(k_2,k_3,\cdots,k_m)\neq(j_1,j_1,\cdots,j_1)}(\gamma_{j_1}^+(\mathcal{A})-a_{j_1k_2\cdots k_m})\right).
\end{align*}
The second inequality holds due to the condition $(c_2)$.

\textbf{Case $c_3$.} When $\gamma_{i_1}^+(\mathcal{A})=\gamma_{j_1}^+(\mathcal{A})=0$, then ,
\begin{align*}
&(a_{i_1i_1\cdots i_1}-\gamma_{i_1}^+(\mathcal{A}))\cdot(a_{j_1j_1\cdots j_1}-\gamma_{j_1}^+(\mathcal{A}))=a_{i_1i_1\cdots i_1}\cdot a_{j_1j_1\cdots j_1}\geq \underline{a}_{i_1i_1\cdots i_1}\cdot\underline{a}_{j_1j_1\cdots j_1}\\
&>\left(\max\left\{0,-\sum\limits_{(k_2,k_3,\cdots,k_m)\neq(i_1,i_1,\cdots,i_1)}\underline{a}_{i_1k_2\cdots k_m}\right\}\right)\cdot\left(\max\left\{0,-\sum\limits_{(k_2,k_3,\cdots,k_m)\neq(j_1,j_1,\cdots,j_1)}\underline{a}_{j_1k_2\cdots k_m}\right\}\right)\\
&=\left(-\sum\limits_{(k_2,k_3,\cdots,k_m)\neq(i_1,i_1,\cdots,i_1)}\underline{a}_{i_1k_2\cdots k_m}\right)\cdot\left(-\sum\limits_{(k_2,k_3,\cdots,k_m)\neq(j_1,j_1,\cdots,j_1)}\underline{a}_{j_1k_2\cdots k_m}\right)\\
&\geq \left(\sum\limits_{(k_2,k_3,\cdots,k_m)\neq(i_1,i_1,\cdots,i_1)}(\gamma_{i_1}^+(\mathcal{A})-a_{i_1k_2\cdots k_m})\right)\cdot\left(\sum\limits_{(k_2,k_3,\cdots,k_m)\neq(j_1,j_1,\cdots,j_1)}(\gamma_{j_1}^+(\mathcal{A})-a_{j_1k_2\cdots k_m})\right).
\end{align*}
The second inequality holds due to the condition $(c_3)$.

Then condition $(c)$ of Definition \ref{def-dbt} is obtained. Therefore, $\mathcal{A}$ is a double B-tensor, thus $\mathcal{A}^I$ is an interval double B-tensor
\end{proof}

Although Theorem \ref{th2} is formally complex, its significance lies in providing a fully computable criterion. Similar to the case of interval B-tensors, we also have corollaries that simplify the verification.

\begin{corollary}\label{c9}
Let $\mathcal{A}^I$ and $\mathcal{B}^I$ be two interval tensors in $T_{m,n}$, here, interval tensor $\mathcal{B}^I$ is defined as Corollary \textup{\ref{c1}} or Corollary \textup{\ref{c2}}. Then $\mathcal{A}^I$ is an interval double B-tensor if and only if $\mathcal{B}^I$ is an interval double B-tensor.
\end{corollary}
\begin{proof}
When interval tensor $\mathcal{B}^I$ is as defined in Corollary \textup{\ref{c2}}, the reasoning for the validity of the equivalence is similar to that in Corollary \textup{\ref{c2}}. We now show that the equivalence holds when  $\mathcal{B}^I$ is as defined in Corollary \textup{\ref{c3}}.

As $\mathcal{B}^I\subset \mathcal{A}^I$, the necessity is evident, we proceed to discuss the sufficiency.

When $K=\emptyset$, $\mathcal{B}^I=\mathcal{A}^I$, is trivial. Suppose $K\neq\emptyset$, by comparing in $\mathcal{B}^I$ and $\mathcal{A}^I$, and since for all $(i_2,i_3,\cdots,i_m)\neq(i_1,i_1,\cdots,i_1)$, $\overline{a}_{i_1i_2\cdots i_m}$ only appears in $(b)$ and $(c)$ of Theorem \ref{th2}, it is sufficient to prove that $(b)$ and $(c)$ of Theorem \ref{th2} hold when $(i_1,l_2,\cdots, l_m)\in K$ for $\mathcal{A}^I$.

Let $(i_1,l_2,\cdots,l_m)\in K$, and $\underline{a}_{i_1i_2'\cdots i_m'}:=\max\{\underline{a}_{i_1i_2\cdots i_m}:(i_2,i_3,\cdots,i_3)\neq(i_1,i_1,\cdots,i_1)\}$, then $\overline{a}_{i_1l_1\cdots i_1}\geq\underline{a}_{i_1i_2'\cdots i_m'}$.
\begin{align*}
\underline{a}_{i_1i_1\cdots i_1}-\overline{a}_{i_1l_2\cdots l_m}&\geq\underline{a}_{i_1i_2\cdots i_m}-\overline{a}_{i_1i_2'\cdots i_m'}\\
&\geq\max\left\{0,\sum\limits_{\substack{(k_2,k_3,\cdots,k_m)\neq(i_1,i_1,\cdots,i_1)\\
      (k_2,k_3,\cdots,k_m)\neq(i_2',i_3',\cdots,i_m')}}(\overline{a}_{i_1i'_2\cdots i'_m}-\underline{a}_{i_1k_2\cdots k_m})\right\}\\
&\geq\max\left\{0,\sum\limits_{\substack{(k_2,k_3,\cdots,k_m)\neq(i_1,i_1,\cdots,i_1)\\
      (k_2,k_3,\cdots,k_m)\neq(l_2,l_3,\cdots,l_m)}}(\overline{a}_{i_1l_2\cdots l_m}-\underline{a}_{i_1k_2\cdots k_m})\right\}.
\end{align*}
The second inequality holds because $(b_1)$ of Theorem \ref{th2} holds for $(i_1,i_2',\cdots,i_m')$. Then $(b_1)$ of Theorem \ref{th2} holds when $(i_1,l_2,\cdots, l_m)\in K$. The proofs of conditions $(b_2)$ and $(c)$ of Theorem \ref{th2} are analogous.
\end{proof}

It is clear that interval double B-tensors are a generalization of interval B-tensors, but the converse is not necessarily true. Then we have the following inclusion relationship.

\begin{proposition}\label{p12}
If $\mathcal{A}^I$ is an interval double B-tensor in $T_{m,n}$, then $\mathcal{A}^I$ is an interval B-tensor.
\end{proposition}

A natural question that arises is: when is an interval double B-tensor actually a strict interval B-tensor? The following proposition provides a clear characterization, revealing its connection with the ``critical" case involving a single row.

\begin{proposition}\label{p13}
If $\mathcal{A}^I$ is an interval double B-tensor in $T_{m,n}$, then precisely one of the following is satisfied.
\begin{itemize}
  \item [(a)]  $\mathcal{A}^I$ is an interval B-tensor, or
  \item [(b)]  there is a unique $j_1\in[n]$ such that
  $$\sum\limits_{j_2,j_3,\cdots,j_m=1}^n\underline{a}_{j_1j_2\cdots j_m}\leq0,$$
  or there is $(l_2,l_3,\cdots,l_m)\neq(j_1,j_1,\cdots,j_1)$ such that
$$\underline{a}_{j_1j_1\cdots j_1}-\overline{a}_{j_1l_2\cdots l_m}=\sum\limits_{\substack{(j_2,j_3,\cdots,j_m)\neq(j_1,j_1,\cdots,j_1)\\
      (j_2,j_3,\cdots,j_m)\neq(l_2,l_3,\cdots,l_m)}}(\overline{a}_{j_1l_2\cdots l_m}-\underline{a}_{j_1j_2\cdots j_m}),$$
      and for all $i_1\in[n]\setminus\{j_1\}$,
$$\sum\limits_{i_2,i_3,\cdots,i_m=1}^n\underline{a}_{i_1i_2\cdots i_m}>0,$$
  or there is $(l_2,l_3,\cdots,l_m)\neq(i_1,i_1,\cdots,i_1)$ such that
$$\underline{a}_{i_1i_1\cdots i_1}-\overline{a}_{i_1l_2\cdots l_m}>\sum\limits_{\substack{(j_2,j_3,\cdots,j_m)\neq(i_1,i_1,\cdots,i_1)\\
      (j_2,j_3,\cdots,j_m)\neq(l_2,l_3,\cdots,l_m)}}(\overline{a}_{i_1l_2\cdots l_m}-\underline{a}_{i_1j_2\cdots j_m}),$$
\end{itemize}
\end{proposition}
\begin{proof}
If $(a)$ holds, that is $\mathcal{A}^I$ is an interval B-tensor. By Corollary \ref{c5}, for all $i_1\in[n]$,
$$\sum\limits_{i_2,i_3,\cdots,i_m=1}^n\underline{a}_{i_1i_2\cdots i_m}>0$$
and for all $(l_2,l_3,\cdots,l_m)\neq(i_1,i_1,\cdots,i_1)$
$$\underline{a}_{i_1i_1\cdots i_1}-\overline{a}_{i_1l_2\cdots l_m}>\sum\limits_{\substack{(i_2,i_3,\cdots,i_m)\neq(i_1,i_1,\cdots,i_1)\\
      (i_2,i_3,\cdots,i_m)\neq(j_2,j_3,\cdots,j_m)}}(\overline{a}_{i_1l_2\cdots l_m}-\underline{a}_{i_1i_2\cdots i_m}).$$
Hence, condition $(b)$ does not holds.

Since $\mathcal{A}^I$ is an interval double B-tensor, then for all $i_1\in[n]$
\begin{align*}
\underline{a}_{i_1i_1\cdots i_1}>\max\{0,\overline{a}_{i_1i_2\cdots i_m}:(i_2,i_3,\cdots,i_m)\neq(i_1,i_1,\cdots,i_1)\}\geq0.
\end{align*}

Suppose $\mathcal{A}^I$ is not an interval B-tensor. Then there exists at least one row that dose not satisfy the conditions of Corollary \ref{c5}.
If there exist two rows $j_1$ and $j_1'$ that do not satisfy the conditions, $j_1\neq j_1'$, then there are the following three cases.

\textbf{Case $1$.} Suppose $\sum\limits_{j_2,j_3,\cdots,j_m=1}^n\underline{a}_{j_1j_2\cdots j_m}\leq0,$
and
$\sum\limits_{j_2,j_3,\cdots,j_m=1}^n\underline{a}_{j_1'j_2\cdots j_m}\leq0.$
Then $$\underline{a}_{j_1j_1\cdots j_1}\leq -\sum\limits_{(j_2,j_3,\cdots,j_m)\neq (j_1,j_1,\cdots,j_1)}\underline{a}_{j_1j_2\cdots j_m},$$
 and
 $$\underline{a}_{j_1'j_1'\cdots j_1'}\leq -\sum\limits_{(j_2,j_3,\cdots,j_m)\neq (j_1',j_1',\cdots,j_1')}\underline{a}_{j_1'j_2\cdots j_m}.$$
Therefore
\begin{align*}
&\underline{a}_{j_1j_1\cdots j_1}\cdot\underline{a}_{j_1'j_1'\cdots j_1'}\leq\left(-\sum\limits_{(j_2,j_3,\cdots,j_m)\neq(j_1,j_1,\cdots,j_1)}\underline{a}_{j_1j_2\cdots j_m}\right)\cdot\left(-\sum\limits_{(j_2,j_3,\cdots,j_m)\neq(j_1',j_1',\cdots,j_1')}\underline{a}_{j_1'j_2\cdots j_m}\right)\\
&=\left(\max\left\{0,-\sum\limits_{(j_2,j_3,\cdots,j_m)\neq(j_1,j_1,\cdots,j_1)}\underline{a}_{j_1j_2\cdots j_m}\right\}\right)\cdot\left(\max\left\{0,-\sum\limits_{(j_2,j_3,\cdots,j_m)\neq(j_1',j_1',\cdots,j_1')}\underline{a}_{j_1'j_2\cdots j_m}\right\}\right).
\end{align*}
In this case, condition $(c_3)$ in Theorem \ref{th2} is not satisfied, which contradicts the fact that $\mathcal{A}^I$ is an interval double B-tensor.

\textbf{Case $2$.} Suppose there is $(l_2,l_3,\cdots,l_m)\neq(j_1,j_1,\cdots,j_1)$ such that
$$\underline{a}_{j_1j_1\cdots j_1}-\overline{a}_{j_1l_2\cdots l_m}=\sum\limits_{\substack{(j_2,j_3,\cdots,j_m)\neq(j_1,j_1,\cdots,j_1)\\
      (j_2,j_3,\cdots,j_m)\neq(l_2,l_3,\cdots,l_m)}}(\overline{a}_{j_1l_2\cdots l_m}-\underline{a}_{j_1j_2\cdots j_m}),$$
and $\sum\limits_{j_2,j_3,\cdots,j_m=1}^n\underline{a}_{j_1'j_2\cdots j_m}\leq0.$
Then $\underline{a}_{j_1'j_1'\cdots j_1'}\leq -\sum\limits_{(j_2,j_3,\cdots,j_m)\neq (j_1',j_1',\cdots,j_1')}\underline{a}_{j_1'j_2\cdots j_m}.$
Therefore
\begin{align*}
&(\underline{a}_{j_1j_1\cdots j_1}-\overline{a}_{j_1l_2\cdots l_m})\cdot\underline{a}_{j_1'j_1'\cdots j_1'}\\
&\leq\left(\sum\limits_{\substack{(j_2,j_3,\cdots,j_m)\neq(j_1,j_1,\cdots,j_1)\\
      (j_2,j_3,\cdots,j_m)\neq(l_2,l_3,\cdots,l_m)}}(\overline{a}_{j_1l_2\cdots l_m}-\underline{a}_{j_1j_2\cdots j_m})\right)\cdot\left(-\sum\limits_{(j_2,j_3,\cdots,j_m)\neq(j_1',j_1',\cdots,j_1')}\underline{a}_{j_1'j_2\cdots j_m}\right)\\
&=\left(\max\left\{0,\sum\limits_{\substack{(j_2,j_3,\cdots,j_m)\neq(j_1,j_1,\cdots,j_1)\\
      (j_2,j_3,\cdots,j_m)\neq(l_2,l_3,\cdots,l_m)}}(\overline{a}_{j_1l_2\cdots l_m}-\underline{a}_{j_1j_2\cdots j_m})\right\}\right)\cdot\\
&~~~\left(\max\left\{0,-\sum\limits_{(j_2,j_3,\cdots,j_m)\neq(j_1',j_1',\cdots,j_1')}\underline{a}_{j_1'j_2\cdots j_m}\right\}\right).
\end{align*}
In this case, condition $(c_2)$ in Theorem \ref{th2} is not satisfied, which contradicts the fact that $\mathcal{A}^I$ is an interval double B-tensor.

\textbf{Case $3$.} Suppose there is $(l_2,l_3,\cdots,l_m)\neq(j_1,j_1,\cdots,j_1)$ such that
$$\underline{a}_{j_1j_1\cdots j_1}-\overline{a}_{j_1l_2\cdots l_m}=\sum\limits_{\substack{(j_2,j_3,\cdots,j_m)\neq(j_1,j_1,\cdots,j_1)\\
      (j_2,j_3,\cdots,j_m)\neq(l_2,l_3,\cdots,l_m)}}(\overline{a}_{j_1l_2\cdots l_m}-\underline{a}_{j_1j_2\cdots j_m}),$$
and there is $(l_2',l_3',\cdots,l_m')\neq(j_1',j_1',\cdots,j_1')$ such that
$$\underline{a}_{j_1'j_1'\cdots j_1'}-\overline{a}_{j_1'l_2'\cdots l_m'}=\sum\limits_{\substack{(j_2,j_3,\cdots,j_m)\neq(j_1',j_1',\cdots,j_1')\\
      (j_2,j_3,\cdots,j_m)\neq(l_2',l_3',\cdots,l_m')}}(\overline{a}_{j_1'l_2'\cdots l_m'}-\underline{a}_{j_1'j_2\cdots j_m}).$$
Therefore
\begin{align*}
&(\underline{a}_{j_1j_1\cdots j_1}-\overline{a}_{j_1l_2\cdots l_m})\cdot(\underline{a}_{j_1'j_1'\cdots j_1'}-\overline{a}_{j_1'l_2'\cdots l_m'})\\
&=\left(\sum\limits_{\substack{(j_2,j_3,\cdots,j_m)\neq(j_1,j_1,\cdots,j_1)\\
      (j_2,j_3,\cdots,j_m)\neq(l_2,l_3,\cdots,l_m)}}(\overline{a}_{j_1l_2\cdots l_m}-\underline{a}_{j_1j_2\cdots j_m})\right)\cdot\left(\sum\limits_{\substack{(j_2,j_3,\cdots,j_m)\neq(j_1',j_1',\cdots,j_1')\\
      (j_2,j_3,\cdots,j_m)\neq(l_2',l_3',\cdots,l_m')}}(\overline{a}_{j_1'l_2'\cdots l_m'}-\underline{a}_{j_1'j_2\cdots j_m})\right)\\
&=\left(\max\left\{0,\sum\limits_{\substack{(j_2,j_3,\cdots,j_m)\neq(j_1,j_1,\cdots,j_1)\\
      (j_2,j_3,\cdots,j_m)\neq(l_2,l_3,\cdots,l_m)}}(\overline{a}_{j_1l_2\cdots l_m}-\underline{a}_{j_1j_2\cdots j_m})\right\}\right)\cdot\\
&~~~\left(\max\left\{0,\sum\limits_{\substack{(j_2,j_3,\cdots,j_m)\neq(j_1',j_1',\cdots,j_1')\\
      (j_2,j_3,\cdots,j_m)\neq(l_2',l_3',\cdots,l_m')}}(\overline{a}_{j_1'l_2'\cdots l_m'}-\underline{a}_{j_1'j_2\cdots j_m})\right\}\right).
\end{align*}
In this case, condition $(c_1)$ in Theorem \ref{th2} is not satisfied, which contradicts the fact that $\mathcal{A}^I$ is an interval double B-tensor.

In sum up, there exists only one row that dose not satisfy the conditions in Corollary \ref{c5}.
\end{proof}

To facilitate verification, we further provide some necessary or sufficient conditions that require checking only a finite number of specific tensors.

\begin{proposition}\label{p14}
Let $\mathcal{A}^I$ be an interval tensor in $T_{m,n}$. For each $i_1\in[n]$, denote $\overline{a}_{i_1k_2^{i_1}\cdots k_m^{i_1}}:=\max\{\overline{a}_{i_1i_2\cdots i_m}:(i_2,i_3,\cdots,i_m)\neq (i_1,i_1,\cdots,i_1)\}$, define $i^{\mathcal{A}_{\max}}=(a_{l_1l_2\cdots l_m})\in T_{m,n}$ as
$$a_{l_1l_2\cdots l_m}=\begin{cases}
\overline{a}_{l_1k_2^{l_1}\cdots k_m^{l_1}},\mbox{ if }l_1\neq i_1\mbox{ and } (l_2,l_3,\cdots,l_m)=(k_2^{l_1},k_3^{l_1},\cdots,k_m^{l_1}),\\
\underline{a}_{l_1l_2\cdots l_m},\mbox{ otherwise}.
\end{cases}$$
$\mathcal{A}^I$ is an interval double B-tensor only if for all $i\in[n]$, $i^{\mathcal{A}_{\max}}$ and $\underline{\mathcal{A}}$ are double B-tensors.
\end{proposition}
\begin{proof}
Since $\mathcal{A}^I$ is an interval double B-tensor, and for all $i\in[n]$, $i_1^{\mathcal{A}_{\max}}\in\mathcal{A}^I$ and $\underline{\mathcal{A}}\in\mathcal{A}^I$, then they are double B-tensors.
\end{proof}

\begin{proposition}\label{p15}
Let $\mathcal{A}^I$ be an interval tensor in $T_{m,n}$. $\mathcal{A}^I$ is an interval double B-tensor only if
\begin{itemize}
  \item [(a)] For each $i_1\in [n]$, $\underline{a}_{i_1i_1\cdots i_1}>\max\{0,\overline{a}_{i_1i_2\cdots i_m}:(i_2,i_3,\cdots,i_m)\neq(i_1,i_1,\cdots,i_1)\}$;
  \item [(b)] For all $i_1\in [n]$, and $\overline{a}_{i_1k_2^{i_1}\cdots k_m^{i_1}}:=\max\{\overline{a}_{i_1i_2\cdots i_m}:(i_2,i_3,\cdots,i_m)\neq (i_1,i_1,\cdots,i_1)\}$:
  \begin{itemize}
  \item [(b$_1$)] if $\overline{a}_{i_1k_2^{i_1}\cdots k_m^{i_1}}>0$, then
  $\underline{a}_{i_1i_1\cdots i_1}-\overline{a}_{i_1k_2^{i_1}\cdots k_m^{i_1}}\geq
      \sum\limits_{\substack{(k_2,k_3,\cdots,k_m)\neq(i_1,i_1,\cdots,i_1)\\
      (k_2,k_3,\cdots,k_m)\neq(k_2^{i_1},k_3^{i_1},\cdots, k_m^{i_1})}}(\overline{a}_{i_1k_2^{i_1}\cdots k_m^{i_1}}-\underline{a}_{i_1k_2\cdots k_m})$;
  \item [(b$_2$)] if $\overline{a}_{i_1k_2^{i_1}\cdots k_m^{i_1}}\leq0$, then
  $\underline{a}_{i_1i_1\cdots i_1}\geq\left(-\sum\limits_{(k_2,k_3,\cdots,k_m)\neq(i_1,i_1,\cdots,i_1)}\underline{a}_{i_1k_2\cdots k_m}\right)$;
  \end{itemize}
  \item [(c)] For all $i_1,j_1\in [n]$, $i_1\neq j_1$, $\overline{a}_{i_1k_2^{i_1}\cdots k_m^{i_1}}:=\max\{\overline{a}_{i_1i_2\cdots i_m}:(i_2,i_3,\cdots,i_m)\neq (i_1,i_1,\cdots,i_1)\}$, and $\overline{a}_{j_1k_2^{j_1}\cdots k_m^{j_1}}:=\max\{\overline{a}_{j_1j_2\cdots j_m}:(j_2,j_3,\cdots,j_m)\neq (j_1,j_1,\cdots,j_1)\}$:
  \begin{itemize}
  \item [(c$_1$)] if $\overline{a}_{i_1k_2^{i_1}\cdots k_m^{i_1}}>0$ and $\overline{a}_{j_1k_2^{j_1}\cdots k_m^{j_1}}>0$, then
  \begin{align*}
  &(\underline{a}_{i_1i_1\cdots i_1}-\overline{a}_{i_1k_2^{i_1}\cdots k_m^{i_1}})\cdot(\underline{a}_{j_1j_1\cdots j_1}-\overline{a}_{j_1k_2^{j_1}\cdots k_m^{j_1}})\\
  &>\left(\sum\limits_{\substack{(k_2,k_3,\cdots,k_m)\neq(i_1,i_1,\cdots,i_1)\\
      (k_2,k_3,\cdots,k_m)\neq(k_2^{i_1},k_3^{i_1},\cdots, k_m^{i_1})}}(\overline{a}_{i_1k_2^{i_1}\cdots k_m^{i_1}}-\underline{a}_{i_1k_2\cdots k_m})\right)\cdot\\
      &~~~\left(\sum\limits_{\substack{(k_2,k_3,\cdots,k_m)\neq(j_1,j_1,\cdots,j_1)\\
      (k_2,k_3,\cdots,k_m)\neq(k_2^{j_1},k_3^{j_1},\cdots, k_m^{j_1})}}(\overline{a}_{j_1k_2^{j_1}\cdots k_m^{j_1}}-\underline{a}_{j_1k_2\cdots k_m})\right);
      \end{align*}
  \item [(c$_2$)] if $\overline{a}_{i_1k_2^{i_1}\cdots k_m^{i_1}}>0$ and $\overline{a}_{j_1k_2^{j_1}\cdots k_m^{j_1}}\leq0$, then
  \begin{align*}
  &(\underline{a}_{i_1i_1\cdots i_1}-\overline{a}_{i_1k_2^{i_1}\cdots k_m^{i_1}})\cdot\underline{a}_{j_1j_1\cdots j_1}\\
  &>\left(\sum\limits_{\substack{(k_2,k_3,\cdots,k_m)\neq(i_1,i_1,\cdots,i_1)\\
      (k_2,k_3,\cdots,k_m)\neq(k_2^{i_1},k_3^{i_1},\cdots, k_m^{i_1})}}(\overline{a}_{i_1k_2^{i_1}\cdots k_m^{i_1}}-\underline{a}_{i_1k_2\cdots k_m})\right)\cdot\left(-\sum\limits_{(k_2,k_3,\cdots,k_m)\neq(j_1,j_1,\cdots,j_1)}\underline{a}_{j_1k_2\cdots k_m}\right);
  \end{align*}
  \item [(c$_3$)] if $\overline{a}_{i_1k_2^{i_1}\cdots k_m^{i_1}}\leq0$ and $\overline{a}_{j_1k_2^{j_1}\cdots k_m^{j_1}}\leq0$, then
  \begin{align*}
  \underline{a}_{i_1i_1\cdots i_1}\cdot\underline{a}_{j_1j_1\cdots j_1}>\left(-\sum\limits_{(k_2,k_3,\cdots,k_m)\neq(i_1,i_1,\cdots,i_1)}\underline{a}_{i_1k_2\cdots k_m}\right)\cdot
      \left(-\sum\limits_{(k_2,k_3,\cdots,k_m)\neq(j_1,j_1,\cdots,j_1)}\underline{a}_{j_1k_2\cdots k_m}\right).
  \end{align*}
      \end{itemize}
\end{itemize}
\end{proposition}

\begin{proposition}\label{p16}
Let $\mathcal{A}^I$ be an interval tensor in $T_{m,n}$, $n\geq3$. For each $i_1\in[n]$, there is $(k_2^{i_1},k_3^{i_1},\cdots,k_m^{i_1})\neq (i_1,i_1,\cdots,i_1)$, for all $(i_2,i_3,\cdots,i_m)\neq(i_1,i_1,\cdots,i_1)$ and $(i_2,i_3,\cdots,i_m)\neq(k_2^{i_1},k_3^{i_1},\cdots,k_m^{i_1})$, such that $\overline{a}_{i_1i_2\cdots i_3}\leq \underline{a}_{i_1k_2^{i_1}\cdots k_m^{i_1}}$. Then
$\mathcal{A}^I$ is an interval double B-tensor if and only if $\mathcal{A}^I$ satisfies the necessary conditions in Proposition \textup{\ref{p14}}.
\end{proposition}
\begin{proof}
``{\bf Necessity.}''  It follows from Proposition \ref{p14}.

``{\bf sufficiency.}'' Based on the assumption, for all $\mathcal{A}\in \mathcal{A}^I$ and $i_1\in[n]$, if $\gamma_{i_1}^+(\mathcal{A})>0$ then $\gamma_{i_1}^+(\mathcal{A})\leq\overline{a}_{i_1k_2^{i_1}\cdots k_m^{i_1}}$ and $\gamma_{i_1}^+(\mathcal{A})=a_{i_1k_2^{i_1}\cdots k_m^{i_1}}$.

Take any $\mathcal{A}\in \mathcal{A}^I$.

As for all $i\in[n]$, $i^{\mathcal{A}_{\max}}\in\mathcal{A}^I$ is a double B-tensor, then $a_{i_1i_1\cdots i_1}\geq \underline{a}_{i_1i_1\cdots i_1}>\max\{0,\overline{a}_{i_1i_2\cdots i_m}:(i_2,i_3,\cdots,i_m)\neq(i_1,i_1,\cdots,i_1)\}\geq\max\{0,a_{i_1i_2\cdots i_m}:(i_2,i_3,\cdots,i_m)\neq(i_1,i_1,\cdots,i_1)\}$ holds for all $i_1\in[n]$. Hence, condition $(a)$ of Definition \ref{def-dbt} is satisfied.

For the verification of condition $(b)$ of Definition \ref{def-dbt}, we distinguish two cases for all $i_1\in[n]$.

\textbf{Case $b_1$.} When $\gamma_{i_1}^+(\mathcal{A})>0$, $\gamma_{i_1}^+(\mathcal{A})=a_{i_1k_2^{i_1}\cdots k_m^{i_1}}>0$. Then
\begin{align*}
a_{i_1i_1\cdots i_1}-\gamma_{i_1}^+(\mathcal{A})&\geq \underline{a}_{i_1i_1\cdots i_1}-\overline{a}_{i_1k_2^{i_1}\cdots k_m^{i_1}}\\
&\geq\sum\limits_{\substack{(k_2,k_3,\cdots,k_m)\neq(i_1,i_1,\cdots,i_1)\\
      (k_2,k_3,\cdots,k_m)\neq(k_2^{i_1},k_3^{i_1},\cdots, k_m^{i_1})}}(\overline{a}_{i_1k_2^{i_1}\cdots k_m^{i_1}}-\underline{a}_{i_1k_2\cdots k_m})\\
&\geq\sum\limits_{\substack{(k_2,k_3,\cdots,k_m)\neq(i_1,i_1,\cdots,i_1)\\
      (k_2,k_3,\cdots,k_m)\neq(k_2^{i_1},k_3^{i_1},\cdots, k_m^{i_1})}}(\gamma_{i_1}^+(\mathcal{A})-a_{i_1k_2\cdots k_m})\\
&=\sum\limits_{(k_2,k_3,\cdots,k_m)\neq(i_1,i_1,\cdots,i_1)}(\gamma_{i_1}^+(\mathcal{A})-a_{i_1k_2\cdots k_m}).
\end{align*}
The second inequality holds because $i^{\mathcal{A}_{\max}}$ for all $i\in[n]\setminus\{ i_1\}$ is a double B-tensor.

\textbf{Case $b_2$.} When $\gamma_{i_1}^+(\mathcal{A})=0$, $0=\gamma_{i_1}^+(\mathcal{A})\geq a_{i_1k_2^{i_1}\cdots k_m^{i_1}}$. Then
\begin{align*}
a_{i_1i_1\cdots i_1}-\gamma_{i_1}^+(\mathcal{A})&\geq \underline{a}_{i_1i_1\cdots i_1}\\
&\geq-\sum\limits_{(k_2,k_3,\cdots,k_m)\neq(i_1,i_1,\cdots,i_1)}\underline{a}_{i_1k_2\cdots k_m}\\
&\geq\sum\limits_{(k_2,k_3,\cdots,k_m)\neq(i_1,i_1,\cdots,i_1)}(\gamma_{i_1}^+(\mathcal{A})-a_{i_1k_2\cdots k_m}).
\end{align*}
The second inequality holds because $i^{\mathcal{A}_{\max}}$ for all $i\in[n]\setminus\{ i_1\}$ is a double B-tensor.

Then condition $(b)$ of Definition \ref{def-dbt} is obtained.

For condition $(c)$ of Definition \ref{def-dbt}, its verification follows a similar logic to that of condition $(b)$ of Definition \ref{def-dbt},  we distinguish there cases for all $i_1,j_1\in[n]$, $i_1\neq j_1$.

\textbf{Case $c_1$.} When $\gamma_{i_1}^+(\mathcal{A})>0$ and $\gamma_{j_1}^+(\mathcal{A})>0$, $\gamma_{i_1}^+(\mathcal{A})=a_{i_1k_2^{i_1}\cdots k_m^{i_1}}>0$ and $\gamma_{j_1}^+(\mathcal{A})=a_{j_1k_2^{j_1}\cdots k_m^{j_1}}>0$. Then
\begin{align*}
&(a_{i_1i_1\cdots i_1}-\gamma_{i_1}^+(\mathcal{A}))\cdot(a_{j_1j_1\cdots j_1}-\gamma_{j_1}^+(\mathcal{A}))\geq (\underline{a}_{i_1i_1\cdots i_1}-\overline{a}_{i_1k_2^{i_1}\cdots k_m^{i_1}})\cdot(\underline{a}_{j_1j_1\cdots j_1}-\overline{a}_{j_1k_2^{j_1}\cdots k_m^{j_1}})\\
&>\left(\sum\limits_{\substack{(k_2,k_3,\cdots,k_m)\neq(i_1,i_1,\cdots,i_1)\\
      (k_2,k_3,\cdots,k_m)\neq(k_2^{i_1},k_3^{i_1},\cdots, k_m^{i_1})}}(\overline{a}_{i_1k_2^{i_1}\cdots k_m^{i_1}}-\underline{a}_{i_1k_2\cdots k_m})\right)\cdot\\
&~~~\left(\sum\limits_{\substack{(k_2,k_3,\cdots,k_m)\neq(j_1,j_1,\cdots,j_1)\\
      (k_2,k_3,\cdots,k_m)\neq(k_2^{j_1},k_3^{j_1},\cdots, k_m^{j_1})}}(\overline{a}_{i_1k_2^{j_1}\cdots k_m^{j_1}}-\underline{a}_{j_1k_2\cdots k_m})\right)\\
&\geq\left(\sum\limits_{\substack{(k_2,k_3,\cdots,k_m)\neq(i_1,i_1,\cdots,i_1)\\
      (k_2,k_3,\cdots,k_m)\neq(k_2^{i_1},k_3^{i_1},\cdots, k_m^{i_1})}}(\gamma_{i_1}^+(\mathcal{A})-a_{i_1k_2\cdots k_m})\right)\cdot\left(\sum\limits_{\substack{(k_2,k_3,\cdots,k_m)\neq(j_1,j_1,\cdots,j_1)\\
      (k_2,k_3,\cdots,k_m)\neq(k_2^{j_1},k_3^{j_1},\cdots, k_m^{j_1})}}(\gamma_{j_1}^+(\mathcal{A})-a_{j_1k_2\cdots k_m})\right)\\
&=\left(\sum\limits_{(k_2,k_3,\cdots,k_m)\neq(i_1,i_1,\cdots,i_1)}(\gamma_{i_1}^+(\mathcal{A})-a_{i_1k_2\cdots k_m})\right)\cdot\left(\sum\limits_{(k_2,k_3,\cdots,k_m)\neq(j_1,j_1,\cdots,j_1)}(\gamma_{j_1}^+(\mathcal{A})-a_{j_1k_2\cdots k_m})\right).
\end{align*}
The second inequality holds because $i^{\mathcal{A}_{\max}}$ for all $i\in[n]\setminus\{ i_1,j_1\}$ is a double B-tensor.

\textbf{Case $c_2$.} When $\gamma_{i_1}^+(\mathcal{A})>0$ and $\gamma_{j_1}^+(\mathcal{A})=0$, $\gamma_{i_1}^+(\mathcal{A})=a_{i_1k_2^{i_1}\cdots k_m^{i_1}}>0$ and $0=\gamma_{j_1}^+(\mathcal{A})\geq a_{j_1k_2^{j_1}\cdots k_m^{j_1}}$. Then
\begin{align*}
&(a_{i_1i_1\cdots i_1}-\gamma_{i_1}^+(\mathcal{A}))\cdot(a_{j_1j_1\cdots j_1}-\gamma_{j_1}^+(\mathcal{A}))\geq (\underline{a}_{i_1i_1\cdots i_1}-\overline{a}_{i_1k_2^{i_1}\cdots k_m^{i_1}})\cdot\underline{a}_{j_1j_1\cdots j_1}\\
&>\left(\sum\limits_{\substack{(k_2,k_3,\cdots,k_m)\neq(i_1,i_1,\cdots,i_1)\\
      (k_2,k_3,\cdots,k_m)\neq(k_2^{i_1},k_3^{i_1},\cdots, k_m^{i_1})}}(\overline{a}_{i_1k_2^{i_1}\cdots k_m^{i_1}}-\underline{a}_{i_1k_2\cdots k_m})\right)\cdot\left(-\sum\limits_{(k_2,k_3,\cdots,k_m)\neq(j_1,j_1,\cdots,j_1)}\underline{a}_{j_1k_2\cdots k_m}\right)\\
&\geq\left(\sum\limits_{\substack{(k_2,k_3,\cdots,k_m)\neq(i_1,i_1,\cdots,i_1)\\
      (k_2,k_3,\cdots,k_m)\neq(k_2^{i_1},k_3^{i_1},\cdots, k_m^{i_1})}}(\gamma_{i_1}^+(\mathcal{A})-a_{i_1k_2\cdots k_m})\right)\cdot\left(\sum\limits_{(k_2,k_3,\cdots,k_m)\neq(j_1,j_1,\cdots,j_1)}(\gamma_{j_1}^+(\mathcal{A})-a_{j_1k_2\cdots k_m})\right)\\
&=\left(\sum\limits_{(k_2,k_3,\cdots,k_m)\neq(i_1,i_1,\cdots,i_1)}(\gamma_{i_1}^+(\mathcal{A})-a_{i_1k_2\cdots k_m})\right)\cdot\left(\sum\limits_{(k_2,k_3,\cdots,k_m)\neq(j_1,j_1,\cdots,j_1)}(\gamma_{j_1}^+(\mathcal{A})-a_{j_1k_2\cdots k_m})\right).
\end{align*}
The second inequality holds because $i^{\mathcal{A}_{\max}}$ for all $i\in[n]\setminus\{ i_1,j_1\}$ is a double B-tensor.

\textbf{Case $c_3$.} When $\gamma_{i_1}^+(\mathcal{A})=\gamma_{j_1}^+(\mathcal{A})=0$, $0=\gamma_{i_1}^+(\mathcal{A})\geq a_{i_1k_2^{i_1}\cdots k_m^{i_1}}>0$ and $0=\gamma_{j_1}^+(\mathcal{A})\geq a_{j_1k_2^{j_1}\cdots k_m^{j_1}}$. Then
\begin{align*}
&(a_{i_1i_1\cdots i_1}-\gamma_{i_1}^+(\mathcal{A}))\cdot(a_{j_1j_1\cdots j_1}-\gamma_{j_1}^+(\mathcal{A}))\geq \underline{a}_{i_1i_1\cdots i_1}\cdot\underline{a}_{j_1j_1\cdots j_1}\\
&>\left(\sum\limits_{(k_2,k_3,\cdots,k_m)\neq(i_1,i_1,\cdots,i_1)}(\overline{a}_{i_1k_2^{i_1}\cdots k_m^{i_1}}-\underline{a}_{i_1k_2\cdots k_m})\right)\cdot\left(-\sum\limits_{(k_2,k_3,\cdots,k_m)\neq(j_1,j_1,\cdots,j_1)}\underline{a}_{j_1k_2\cdots k_m}\right)\\
&\geq\left(\sum\limits_{(k_2,k_3,\cdots,k_m)\neq(i_1,i_1,\cdots,i_1)}(\gamma_{i_1}^+(\mathcal{A})-a_{i_1k_2\cdots k_m})\right)\cdot\left(\sum\limits_{(k_2,k_3,\cdots,k_m)\neq(j_1,j_1,\cdots,j_1)}(\gamma_{j_1}^+(\mathcal{A})-a_{j_1k_2\cdots k_m})\right).
\end{align*}
The second inequality holds because $i^{\mathcal{A}_{\max}}$ for all $i\in[n]\setminus\{ i_1,j_1\}$ is a double B-tensor.

Then condition $(b)$ of Definition \ref{def-dbt} is obtained.

In summary, every $\mathcal{A}\in\mathcal{A}^I$ is a double B-tensor, and thus $\mathcal{A}^I$ is an interval double B-tensor.
\end{proof}

\begin{proposition}\label{p17}
If $\mathcal{A}^I$ be an interval Z tensor in $T_{m,n}$. Then $\mathcal{A}^I$ is an interval double B-tensor if and only if $\underline{\mathcal{A}}$ is a double B-tensor.
\end{proposition}
\begin{proof}
``{\bf Necessity.}'' It follows from $\underline{\mathcal{A}}\in \mathcal{A}^I$.

``{\bf sufficiency.}'' Let $\mathcal{A}\in \mathcal{A}^I$.

$(a)$ Since $\mathcal{A}^I$ is an interval Z tensor and $\underline{\mathcal{A}}$ is a double B-tensor, then for all $i_1\in[n]$, $a_{i_1i_1\cdots i_1}\geq \underline{a}_{i_1i_1\cdots i_1}>\max\{0,\underline{a}_{i_1i_2\cdots i_m}:(i_2,i_3,\dots,i_m)\neq(i_1,i_1,\cdots,i_1)\}=0=\max\{0,a_{i_1i_2\cdots i_m}:(i_2,i_3,\dots,i_m)\neq(i_1,i_1,\cdots,i_1)\}$.

$(b)$ Let $i_1\in[n]$, then
\begin{align*}
a_{i_1i_1\cdots i_1}-\gamma_{i_1}^+(\mathcal{A})&\geq \underline{a}_{i_1i_1\cdots i_1}-0\\
&\geq\sum\limits_{(k_2,k_3,\cdots,k_m)\neq(i_1,i_1,\cdots,i_1)}(0-\underline{a}_{i_1k_2\cdots k_m})\\
&\geq\sum\limits_{(k_2,k_3,\cdots,k_m)\neq(i_1,i_1,\cdots,i_1)}(\gamma_{i_1}^+(\mathcal{A})-a_{i_1k_2\cdots k_m}).
\end{align*}
The second inequality holds because $\underline{\mathcal{A}}$ is a double B-tensor and a Z tensor.

$(c)$ Let $i_1,j_1\in[n]$, $i_1\neq j_1$, then
\begin{align*}
&(a_{i_1i_1\cdots i_1}-\gamma_{i_1}^+(\mathcal{A}))\cdot(a_{j_1j_1\cdots j_1}-\gamma_{j_1}^+(\mathcal{A}))\geq (\underline{a}_{i_1i_1\cdots i_1}-0)\cdot(\underline{a}_{j_1j_1\cdots j_1}-0)\\
&>\left(\sum\limits_{(k_2,k_3,\cdots,k_m)\neq(i_1,i_1,\cdots,i_1)}(0-\underline{a}_{i_1k_2\cdots k_m})\right)\cdot\left(\sum\limits_{(k_2,k_3,\cdots,k_m)\neq(j_1,j_1,\cdots,j_1)}(0-\underline{a}_{j_1k_2\cdots k_m})\right)\\
&\geq\left(\sum\limits_{(k_2,k_3,\cdots,k_m)\neq(i_1,i_1,\cdots,i_1)}(\gamma_{i_1}^+(\mathcal{A})-a_{i_1k_2\cdots k_m})\right)\cdot\left(\sum\limits_{(k_2,k_3,\cdots,k_m)\neq(j_1,j_1,\cdots,j_1)}(\gamma_{j_1}^+(\mathcal{A})-a_{j_1k_2\cdots k_m})\right).
\end{align*}
The second inequality holds because $\underline{\mathcal{A}}$ is a double B-tensor and a Z tensor.

In summary, every $\mathcal{A}\in\mathcal{A}^I$ is a double B-tensor, and thus $\mathcal{A}^I$ is an interval double B-tensor.
\end{proof}

\begin{proposition}\label{p18}
Let $\mathcal{A}^I$ be an interval tensor in $T_{m,n}$. For each $i_1\in[n]$, denote $\overline{a}_{i_1k_2^{i_1}\cdots k_m^{i_1}}:=\max\{\overline{a}_{i_1i_2\cdots i_m}:(i_2,i_3,\cdots,i_m)\neq (i_1,i_1,\cdots,i_1)\}$, $\underline{a}_{i_1\underline{k}_2^{i_1}\cdots \underline{k}_m^{i_1}}:=\max\{\underline{a}_{i_1i_2\cdots i_m}:(i_2,i_3,\cdots,i_m)\neq (i_1,i_1,\cdots,i_1)\}$. Define $\widehat{\mathcal{A}}=(\widehat{a}_{l_1l_2\cdots l_m})\in T_{m,n}$ as
$$\widehat{a}_{l_1l_2\cdots l_m}=\begin{cases}
\overline{a}_{l_1k_2^{l_1}\cdots k_m^{l_1}},\mbox{ if }(l_2,l_3,\cdots,l_m)=(k_2^{l_1},k_3^{l_1},\cdots,k_m^{l_1}),\\
\underline{a}_{l_1l_2\cdots l_m},\mbox{ if }l_1=l_2=\cdots=l_m,\\
\min\{\underline{a}_{l_1l_2\cdots l_m},\underline{a}_{l_1k_2^{l_1}\cdots k_m^{l_1}}\},\mbox{ otherwise}.
\end{cases}$$
Then $\mathcal{A}^I$ is an interval double B-tensors, when for all $i_1\in[n]$, $\underline{a}_{i_1\underline{k}_2^{i_1}\cdots \underline{k}_m^{i_1}}\geq0$ and $\widehat{\mathcal{A}}$ is a double B-tensor.
\end{proposition}
\begin{proof}
Let $\mathcal{A}\in\mathcal{A}^I$, $i_1,j_1\in[n]$, $i_1\neq j_1$. Then $\overline{a}_{i_1k_2^{i_1}\cdots k_m^{i_1}}\geq\overline{a}_{i_1\underline{k}_2^{i_1}\cdots \underline{k}_m^{i_1}}\geq\underline{a}_{i_1\underline{k}_2^{i_1}\cdots \underline{k}_m^{i_1}}\geq0$.
And $a_{i_1i_1\cdots i_1}\geq\underline{a}_{i_1i_1\cdots i_1}>\max\{0,\overline{a}_{i_1k_2^{i_1}\cdots k_m^{i_1}}\}=\overline{a}_{i_1k_2^{i_1}\cdots k_m^{i_1}}\geq \max\{0,a_{i_1i_2\cdots i_m}:(i_2,i_3,\cdots,i_m)\neq(i_1,i_1,\cdots,i_1)\}$ since $\widehat{\mathcal{A}}$ is a double B-tensor. Therefore, for all $\mathcal{A}\in\mathcal{A}^I$ condition $(a)$ of Definition \ref{def-dbt} is satisfied.

For all $i_1\in[n]$, $a_{i_1l_2^{i_1}\cdots l_m^{i_m}}:=\max\{a_{i_1i_2\cdots i_m}:(i_2,i_3,\cdots,i_m)\neq(i_1,i_1,\cdots,i_1)\}$, then $\gamma_{i_1}^+(\mathcal{A})=a_{i_1l_2^{i_1}\cdots l_m^{i_m}}$ as $\max\{a_{i_1i_2\cdots i_m}:(i_2,i_3,\cdots,i_m)\neq(i_1,i_1,\cdots,i_1)\}\geq\underline{a}_{i_1\underline{k}_2^{i_1}\cdots \underline{k}_m^{i_1}}\geq0$. Therefore
\begin{align*}
a_{i_1i_1\cdots i_1}-\gamma_{i_1}^+(\mathcal{A})&\geq \underline{a}_{i_1i_1\cdots i_1}-\overline{a}_{i_1k_2^{i_1}\cdots k_m^{i_1}}=\widehat{a}_{i_1i_1\cdots i_1}-\gamma_{i_1}^+(\widehat{\mathcal{A}})\\
&\geq\sum\limits_{(k_2,k_3,\cdots,k_m)\neq(i_1,i_1,\cdots,i_1)}(\gamma_{i_1}^+(\widehat{\mathcal{A}})-\widehat{a}_{i_1k_2\cdots k_m})\\
&=\sum\limits_{\substack{(k_2,k_3,\cdots,k_m)\neq(i_1,i_1,\cdots,i_1)\\
      (k_2,k_3,\cdots,k_m)\neq(k_2^{i_1},k_3^{i_1},\cdots, k_m^{i_1})}}(\overline{a}_{i_1k_2^{i_1}\cdots k_m^{i_1}}-\min\{\underline{a}_{l_1l_2\cdots l_m},\underline{a}_{l_1k_2^{l_1}\cdots k_m^{l_1}}\})\\
&\geq\sum\limits_{(k_2,k_3,\cdots,k_m)\neq(i_1,i_1,\cdots,i_1)}(a_{i_1l_2^{i_1}\cdots l_m^{i_m}}-a_{i_1k_2\cdots k_m})\\
&=\sum\limits_{(k_2,k_3,\cdots,k_m)\neq(i_1,i_1,\cdots,i_1)}(\gamma_{i_1}^+(\mathcal{A})-a_{i_1k_2\cdots k_m}),
\end{align*}
which implies that for all $\mathcal{A}\in\mathcal{A}^I$ condition $(b)$ of Definition \ref{def-dbt} is satisfied. Conditions $(c)$ of Definition \ref{def-dbt} is discussed analogously.
\begin{align*}
&(a_{i_1i_1\cdots i_1}-\gamma_{i_1}^+(\mathcal{A}))\cdot(a_{j_1j_1\cdots j_1}-\gamma_{j_1}^+(\mathcal{A}))\geq (\underline{a}_{i_1i_1\cdots i_1}-\overline{a}_{i_1k_2^{i_1}\cdots k_m^{i_1}})\cdot(\underline{a}_{j_1j_1\cdots j_1}-\overline{a}_{j_1k_2^{j_1}\cdots k_m^{j_1}})\\
&=(\widehat{a}_{i_1i_1\cdots i_1}-\gamma_{i_1}^+(\widehat{\mathcal{A}}))\cdot(\widehat{a}_{j_1j_1\cdots j_1}-\gamma_{j_1}^+(\widehat{\mathcal{A}}))\\
&>\left(\sum\limits_{(k_2,k_3,\cdots,k_m)\neq(i_1,i_1,\cdots,i_1)}(\gamma_{i_1}^+(\widehat{\mathcal{A}})-\widehat{a}_{i_1k_2\cdots k_m})\right)\cdot\left(\sum\limits_{(k_2,k_3,\cdots,k_m)\neq(j_1,j_1,\cdots,j_1)}(\gamma_{j_1}^+(\widehat{\mathcal{A}})-\widehat{a}_{j_1k_2\cdots k_m})\right)\\
&=\left(\sum\limits_{\substack{(k_2,k_3,\cdots,k_m)\neq(i_1,i_1,\cdots,i_1)\\
      (k_2,k_3,\cdots,k_m)\neq(k_2^{i_1},k_3^{i_1},\cdots, k_m^{i_1})}}(\overline{a}_{i_1k_2^{i_1}\cdots k_m^{i_1}}-\min\{\underline{a}_{l_1l_2\cdots l_m},\underline{a}_{l_1k_2^{l_1}\cdots k_m^{l_1}}\})\right)\cdot\\
&~~~\left(\sum\limits_{\substack{(k_2,k_3,\cdots,k_m)\neq(j_1,j_1,\cdots,j_1)\\
      (k_2,k_3,\cdots,k_m)\neq(k_2^{j_1},k_3^{j_1},\cdots, k_m^{j_1})}}(\overline{a}_{j_1k_2^{j_1}\cdots k_m^{j_1}}-\min\{\underline{a}_{l_1l_2\cdots l_m},\underline{a}_{l_1k_2^{l_1}\cdots k_m^{l_1}}\})\right)\\
&\geq\left(\sum\limits_{(k_2,k_3,\cdots,k_m)\neq(i_1,i_1,\cdots,i_1)}(a_{i_1l_2^{i_1}\cdots l_m^{i_m}}-a_{i_1k_2\cdots k_m})\right)\cdot\left(\sum\limits_{(k_2,k_3,\cdots,k_m)\neq(j_1,j_1,\cdots,j_1)}(a_{j_1l_2^{j_1}\cdots l_m^{j_m}}-a_{j_1k_2\cdots k_m})\right)\\
&=\left(\sum\limits_{(k_2,k_3,\cdots,k_m)\neq(i_1,i_1,\cdots,i_1)}(\gamma_{i_1}^+(\mathcal{A})-a_{i_1k_2\cdots k_m})\right)\cdot\left(\sum\limits_{(k_2,k_3,\cdots,k_m)\neq(j_1,j_1,\cdots,j_1)}(\gamma_{j_1}^+(\mathcal{A})-a_{j_1k_2\cdots k_m})\right),
\end{align*}
which implies that for all $\mathcal{A}\in\mathcal{A}^I$ condition $(c)$ of Definition \ref{def-dbt} is satisfied.

Hence, every $\mathcal{A}\in\mathcal{A}^I$ is a double B-tensor, and thus $\mathcal{A}^I$ is an interval double B-tensor.
\end{proof}

For the special case of circulant tensors, interval double B-tensors and interval B-tensors are equivalent, and the criterion can be reduced to conditions that apply only to the first row.

\begin{proposition}\label{p20}
Let $\mathcal{A}^I$ be an interval tensor in $T_{m,n}$, where $\underline{\mathcal{A}}$ and $\overline{\mathcal{A}}$ are circulant. Then the equivalence holds for the following assertions.
\begin{itemize}
  \item [(a)] $\mathcal{A}^I$ is an interval double B-tensor.
  \item [(b)] $\mathcal{A}^I$ is an interval B-tensor.
  \item [(c)] \begin{itemize}
                \item [(c$_1$)] $\underline{a}_{11\cdots1}>-\sum\limits_{(i_2,i_3,\cdots,i_m)\neq(1,1,\cdots,1)}\underline{a}_{1i_2\cdots i_m}$,
                \item [(c$_2$)] $\underline{a}_{11\cdots1}-\overline{a}_{1j_2\cdots j_m}>\sum\limits_{\substack{(i_2,i_3,\cdots,i_m)\neq(1,1\cdots,1)\\
      (i_2,i_3,\cdots,i_m)\neq(j_2,j_3,\cdots, j_m)}}(\overline{a}_{1i_2\cdots i_m}-\underline{a}_{1j_2\cdots j_m})$ holds for all $(j_2,j_3,\cdots,j_m)\neq(1,1,\cdots,1)$.
              \end{itemize}
  \end{itemize}
\end{proposition}
\begin{proof}
``$(a)\Rightarrow(b)$''. For any $i_1\in[n]$ and $(i_2,i_3,\cdots,i_m)\neq (i_1,i_1,\cdots,i_m)$, let $j_l\in[n]$, $j_l=k_l+1(\mbox{mod }n)$, $l\in[m]$. Since $\mathcal{A}^I$ is an interval double B-tensor, then
\begin{align*}
&a_{i_1i_1\cdots i_1}a_{j_1j_1\cdots j_1}\\
&> \left(\max\left\{0,-\sum\limits_{(k_2,k_3,\cdots,k_m)\neq(i_1,i_1,\cdots,i_1)}\underline{a}_{i_1k_2\cdots k_m}\right\}\right)\cdot\left(\max\left\{0,-\sum\limits_{(k_2,k_3,\cdots,k_m)\neq(j_1,j_1,\cdots,j_1)}\underline{a}_{i_1k_2\cdots k_m}\right\}\right).
\end{align*}
As $\underline{\mathcal{A}}$ is circulant, then
\begin{align*}
a_{i_1i_1\cdots i_1}^2> \left(\max\left\{0,-\sum\limits_{(k_2,k_3,\cdots,k_m)\neq(i_1,i_1,\cdots,i_1)}\underline{a}_{i_1k_2\cdots k_m}\right\}\right)^2.
\end{align*}
Thus,
\begin{align*}
a_{i_1i_1\cdots i_1}=|a_{i_1i_1\cdots i_1}|> \left|\max\left\{0,-\sum\limits_{(k_2,k_3,\cdots,k_m)\neq(i_1,i_1,\cdots,i_1)}\underline{a}_{i_1k_2\cdots k_m}\right\}\right|\geq-\sum\limits_{(k_2,k_3,\cdots,k_m)\neq(i_1,i_1,\cdots,i_1)}\underline{a}_{i_1k_2\cdots k_m}.
\end{align*}
Therefore, $\sum\limits_{k_2,k_3,\cdots,k_m=1}^n\underline{a}_{i_1k_2\cdots k_m}>0$, which implies that $\mathcal{A}^I$ satisfies condition $(a)$ of Corollary \ref{c5}. From Theorem \ref{th2},
\begin{align*}
&(\underline{a}_{i_1i_1\cdots i_1}-\overline{a}_{i_1i_2\cdots i_m})\cdot(\underline{a}_{j_1j_1\cdots j_1}-\overline{a}_{j_1j_2\cdots j_m})\\
&>\left(\max\left\{0,\sum\limits_{\substack{(k_2,k_3,\cdots,k_m)\neq(i_1,i_1,\cdots,i_1)\\
      (k_2,k_3,\cdots,k_m)\neq(i_2,i_3,\cdots,i_m)}}(\overline{a}_{i_1i_2\cdots i_m}-\underline{a}_{i_1k_2\cdots k_m})\right\}\right)\cdot\\
&~~~\left(\max\left\{0,\sum\limits_{\substack{(k_2,k_3,\cdots,k_m)\neq(j_1,j_1,\cdots,j_1)\\
      (k_2,k_3,\cdots,k_m)\neq(j_2,j_3,\cdots,j_m)}}(\overline{a}_{j_1j_2\cdots j_m}-\underline{a}_{j_1k_2\cdots k_m})\right\}\right),
\end{align*}
then according to $\underline{\mathcal{A}}$ and $\overline{\mathcal{A}}$ are circulant,
\begin{align*}
(\underline{a}_{i_1i_1\cdots i_1}-\overline{a}_{i_1i_2\cdots i_m})^2>\left(\max\left\{0,\sum\limits_{\substack{(k_2,k_3,\cdots,k_m)\neq(i_1,i_1,\cdots,i_1)\\
      (k_2,k_3,\cdots,k_m)\neq(i_2,i_3,\cdots,i_m)}}(\overline{a}_{i_1i_2\cdots i_m}-\underline{a}_{i_1k_2\cdots k_m})\right\}\right)^2.
\end{align*}
Therefore,
\begin{align*}
\underline{a}_{i_1i_1\cdots i_1}-\overline{a}_{i_1i_2\cdots i_m}&=|\underline{a}_{i_1i_1\cdots i_1}-\overline{a}_{i_1i_2\cdots i_m}|>\left|\max\left\{0,\sum\limits_{\substack{(k_2,k_3,\cdots,k_m)\neq(i_1,i_1,\cdots,i_1)\\
      (k_2,k_3,\cdots,k_m)\neq(i_2,i_3,\cdots,i_m)}}(\overline{a}_{i_1i_2\cdots i_m}-\underline{a}_{i_1k_2\cdots k_m})\right\}\right|\\
      &\geq\sum\limits_{\substack{(k_2,k_3,\cdots,k_m)\neq(i_1,i_1,\cdots,i_1)\\
      (k_2,k_3,\cdots,k_m)\neq(i_2,i_3,\cdots,i_m)}}(\overline{a}_{i_1i_2\cdots i_m}-\underline{a}_{i_1k_2\cdots k_m}).
\end{align*}
This shows that $\mathcal{A}^I$ satisfies condition $(b)$ of Corollary \ref{c5}, thus $\mathcal{A}^I$ is an interval B-tensor.

``$(b)\Rightarrow(a)$'' Trivial.

``$(b)\Rightarrow(c)$'' Since $\mathcal{A}^I$ is an interval B-tensor, the conclusion follows directly form Corollary \ref{c5}.

``$(c)\Rightarrow(b)$'' As $\underline{\mathcal{A}}$ is circulant, by $(c_1)$ we can obtain that for all $i_1\in[n]$, $\sum\limits_{i_2,i_3,\cdots,i_m=1}^n\underline{a}_{i_1,i_2\cdots i_m}>0$.

And since $\underline{\mathcal{A}}$ also is circulant, then for all $(i_2,i_3,\cdots,i_m)\neq(1,1,\cdots,1)$, there exists $(i_2',i_3',\cdots,i_m')\neq(i_1,i_1,\cdots,i_1)$ such that
\begin{align*}
\underline{a}_{11\cdots 1}-\overline{a}_{1i_2\cdots i_m}=\underline{a}_{i_1i_1\cdots i_1}-\overline{a}_{i_1i_2'\cdots i_m'},
\end{align*}
and
\begin{align*}
\sum\limits_{\substack{(k_2,k_3,\cdots,k_m)\neq(1,1,\cdots,1)\\
      (k_2,k_3,\cdots,k_m)\neq(i_2,i_3,\cdots,i_m)}}(\overline{a}_{1i_2\cdots i_m}-\underline{a}_{1k_2\cdots k_m})=\sum\limits_{\substack{(k_2,k_3,\cdots,k_m)\neq(i_1,i_1,\cdots,i_1)\\
      (k_2,k_3,\cdots,k_m)\neq(i_2',i_3',\cdots,i_m')}}(\overline{a}_{i_1i_2'\cdots i_m'}-\underline{a}_{i_1k_2\cdots k_m}).
\end{align*}
 By $(c_2)$ we can obtain that for all $i_1\in[n]$, and $(i_2,i_3,\cdots,i_m)\neq(i_1,i_1,\cdots,i_1)$, such that
 \begin{align*}
\underline{a}_{i_1i_1\cdots i_1}-\overline{a}_{i_1i_2\cdots i_m}>\sum\limits_{\substack{(k_2,k_3,\cdots,k_m)\neq(i_1,i_1,\cdots,i_1)\\
      (k_2,k_3,\cdots,k_m)\neq(i_2,i_3,\cdots,i_m)}}(\overline{a}_{i_1i_2\cdots i_m}-\underline{a}_{i_1k_2\cdots k_m}).
\end{align*}
According to Corollary \ref{c5}, $\mathcal{A}^I$ is an interval B-tensor.
\end{proof}

\begin{example}
Interval tensor $\mathcal{B}^I = [\underline{\mathcal{B}}, \overline{\mathcal{B}}] \in T_{3,2}$ with
$\underline{b}_{111}= \underline{b}_{222}=6$, $\underline{b}_{112}=\underline{b}_{121}=\underline{b}_{211}= \underline{b}_{122}=\underline{b}_{212}=\underline{b}_{221}=0$, $b_{111}=b_{222}=7$, and $b_{112}=b_{121}=b_{211}=b_{122}=b_{212}=b_{221}=1$.
\end{example}
By check the conditions in Theorem \ref{th2}, we can prove that the interval tensor $\mathcal{B}^I$ is an interval double B-tensor. For example:

$$\underline{b}_{111}=6 > \max\{0, \overline{b}_{112},\overline{b}_{121},\overline{b}_{122}\} = \max\{0,1,1,1\} = 1,$$
then ($a$) of Theorem \ref{th2} is satisfied for $i_1=1$.
$$\underline{b}_{111} - \overline{b}_{112} = 6-1 = 5 \geq \sum\limits_{\substack{(k_2,k_3)\neq(1,1)\\ \neq(1,2)}} (\overline{b}_{112} - \underline{b}_{1k_2k_3}) = (1-0)+(1-0) = 2.$$
then ($b_1$) of Theorem \ref{th2} is satisfied for $i_1=1$ and $(i_2,i_3)=(1,2)$,
$$(\underline{b}_{111}-\overline{b}_{112})(\underline{b}_{222}-\overline{b}_{221}) = (6-1)(6-1) = 25
> \bigl((1-0)+(1-0)\bigr)\bigl((1-0)+(1-0)\bigr) = 4,$$
then ($c_1$) of Theorem \ref{th2} is satisfied for $i_1=1$, $j_1=2$.

\section{Conclusions}
This paper systematically develops a theoretical framework for interval B-tensors and interval double B-tensors, providing a structured approach to analyze tensors subject to interval uncertainty. Firstly, we introduce definitions for these interval tensor classes and then establish a series of verifiable criteria that depend solely on the extreme point tensors, clarifying their intrinsic connections with key structures such as strictly diagonally dominated tensors, interval Z-tensors, and interval P-tensors. Furthermore, for special structures like circulant tensors, simplified checking conditions are derived, which significantly reduce computational complexity. Under the assumptions of even order and symmetry, we prove that both interval B-tensors and interval double B-tensors ensure the interval P-tensor property, thereby establishing theoretical links to positive definiteness and complementarity problems. These results extend classical interval matrix theory to the higher-order tensor setting and offer both theoretical foundations and practical tools for analyzing tensor structures under uncertainty.

This study lays a foundation for further exploration of interval structured tensors. Future work may focus on developing efficient verification algorithms for large-scale problems and extending the present theory to other important classes of structured interval tensors. Furthermore, applications in areas such as robust polynomial optimization, uncertain multilinear systems, and data-driven machine learning warrant in-depth investigation.

\section*{Competing interest}
The author declares that he has no known competing financial interests or personal relationships that could have appeared to influence the work reported in this paper.
\section*{Availability of data and materials}
This manuscript has no associated data or the data will not be deposited. [Author's comment: This is a theoretical study and there are no external data associated with the manuscript].
\section*{Funding}
This  work was supported by the National Natural Science Foundation of P.R. China (Grant No.12171064).




\end{document}